\documentclass[reqno]{amsart}

\usepackage{enumerate,xcolor,multirow,dsfont,amsthm,fixltx2e}

\newcommand*{\abs}[1]{\left|#1\right|}
\newcommand*{\set}[1]{\left\{#1\right\}}
\newcommand{\ve}{\mathbf}
\newcommand{\eps}{\varepsilon}
\renewcommand{\P}{\mathbf P}
\newcommand{\E}{\mathbf E}
\newcommand{\R}{\mathbb R}

\newcommand{\T}{\mathbf T}
\newcommand{\Wt}{\widetilde W}

\newtheorem{theorem}{Theorem}[section]
\newtheorem{lemma}[theorem]{Lemma}

\newtheorem{corollary}[theorem]{Corollary}
\theoremstyle{remark}

\newtheorem{rem}{Remark}
\theoremstyle{definition}

\allowdisplaybreaks[1]

\begin{document}
\title[Asymptotic growth of trajectories of mBm]{Asymptotic growth of trajectories of  multifractional Brownian motion, with statistical applications to drift parameter estimation}

\author[M. Dozzi]{Marco Dozzi}
\address{Institut \'Elie Cartan,
Universit\'e du Lorraine,
Campus-des-Aiguillettes,
BP 70239,
F-54506 Vandoeuvre-les-Nancy Cedex,
France}
\email{marco.dozzi@univ-lorraine.fr}

\author[Y. Kozachenko]{Yuriy Kozachenko}
\address{Department of Probability, Statistics and Actuarial Mathematics,
Mechanics and Mathematics Faculty,
Taras Shevchenko National University of Kyiv,
64, Volodymyrs'ka St.,
01601 Kyiv, Ukraine}
\email{yvk@univ.kiev.ua}

\author[Y. Mishura]{Yuliya Mishura}
\address{Department of Probability, Statistics and Actuarial Mathematics,
Mechanics and Mathematics Faculty,
Taras Shevchenko National University of Kyiv,
64, Volodymyrs'ka St.,
01601 Kyiv, Ukraine}
\email{myus@univ.kiev.ua}

\author[K. Ralchenko]{Kostiantyn Ralchenko}
\address{Department of Probability, Statistics and Actuarial Mathematics,
Mechanics and Mathematics Faculty,
Taras Shevchenko National University of Kyiv,
64, Volodymyrs'ka St.,
01601 Kyiv, Ukraine}
\email{k.ralchenko@gmail.com}


\begin{abstract}
We construct the least-square estimator for the unknown drift parameter in the multifractional Ornstein--Uhlenbeck model and establish
 its strong consistency in the non-ergodic case. The proofs are based on the asymptotic bounds with probability 1 for the rate of the growth of the trajectories of multifractional Brownian motion (mBm) and of some other functionals of mBm, including increments and fractional derivatives. As the auxiliary results   having independent interest, we produce the   asymptotic bounds with probability 1 for the rate of the growth of the trajectories of the general Gaussian process and some functionals of it, in terms of the covariance function of its increments.
\end{abstract}

\maketitle

\section{Introduction} The goal of the present paper is twofold. First, we get the asymptotic bounds with probability 1 for the rate of the growth of the trajectories of multifractional Brownian motion (mBm) and of some other functionals of mBm, including increments and fractional derivatives. Second, we apply these bounds to construct consistent estimators of the unknown drift parameter in the linear and Ornstein--Uhlenbeck model involving mBm. As the auxiliary results having independent interest, we produce the   asymptotic bounds with probability 1 for the rate of the growth of the trajectories of the general Gaussian process and some functionals of it, in terms of the covariance function of its increments. The results obtained generalize the  respective results concerning
asymptotic bounds with probability 1 for the rate of the growth of the trajectories of  fractional Brownian motion (fBm)  from~\cite{KMM} and numerous results concerning consistent estimators of the unknown drift parameter in the linear and Ornstein--Uhlenbeck model involving fBm. The  extended survey of these results is contained, e.g.,  in the paper~\cite{kumirase}. The methods of constructing the estimators and their properties in the fractional Brownian case   essentially depend on factors such as the value of Hurst index $H$, more precisely,  cases  $H>1/2$ and  $H<1/2$ differ substantially; on the sign of unknown drift parameter $\theta$ and  also on whether the continuous and discrete observations. The MLE estimators of the unknown drift parameter for fractional Ornstein--Uhlenbeck process with $H\geq 1/2$ and   any $\theta\in\R$ were constructed in \cite{kllebr} with the help of so called Molchan fundamental martingale. The same estimator was studied in the paper \cite{tuvi} for $H<1/2$. In the paper \cite{hunu} the analog of the least-square estimator of the form $\widehat{\theta}_T=\frac{\int_0^TX_tdX_t}{\int_0^TX^2_tdt}$ was constructed for $H\geq 1/2$ and $\theta<0$ (the ergodic case) was studied in the supposition that the integral $\int_0^TX_tdX_t$ is the divergence-type one. As an alternative estimator, $\left(\frac{1}{T}\int_0^TX^2_tdt\right)^{\frac{1}{2H}}$ was proposed, and its properties are essentially based on the ergodic properties of the fractional Ornstein-Uhlenbeck process with negative drift.  In the papers  \cite{kumirase}, \cite{mira} and \cite{mirasesh} the discretized estimators were proposed. Another approach to discretization was studied in \cite{azmo}. Note that the MLE is hardly discretized because of singular kernels and one should choose the nonstandard estimators for discretization, that was done in these papers. Mention also that the discretized estimator proposed in \cite{kumirase} for $H<1/2$, in reality works properly only for $\theta\geq 0$ and   apparently does not work in   the ergodic case. In general, the problem of the discretization for $H<1/2$ and $\theta<0$ is open. Contrary to fractional case, multifractional Ornstein--Uhlenbeck model was not considered in its entirety, even though these models are gaining increasing popularity now. We can mention only the paper \cite{EMESO}, where the least square estimator is studied for the non-ergodic Ornstein--Uhlenbeck process with some special  Gaussian process, the case that includes not only fractional but, e.g.,   subfractional and bi-fractional Brownian motions.   In the present  paper we consider Ornstein--Uhlenbeck multifractional processes   when the index $H_t$ of multifractionality is  bounded from below by some constant exceeding $\frac{1}{2}$, and observations are continuous in time. We consider non-ergodic case, because the asymptotical bounds for the growth of Ornstein--Uhlenbeck multifractional process work properly in the non-ergodic case. The problem of the drift parameter estimation in the multifractional  Ornstein--Uhlenbeck model is still open. Note that the linear model with unknown drift parameter is considered, and the  properties of the estimator are based on the asymptotic growth of the trajectories of mBm.
The paper is organized as follows.
Section~\ref{sec:Gauss} contains auxiliary
results for the asymptotic growth of Gaussian processes defined on arbitrary parameter set, on the half-axis, and in the strip on the plane.
In Section~\ref{sec:mBm} we  establish the asymptotic  growth with probability 1 of mBm and its increments.
In Section~\ref{sec:estimation} we investigate two statistical models with mBm: the linear model and the multifractional  Ornstein--Uhlenbeck process.
For these models we propose estimators for an unknown drift parameter and prove their strong consistency.

\section{ Exponential maximal bounds and asymptotic growth of trajectories of Gaussian processes}\label{sec:Gauss}
Let $\T$ be a parameter set and
$X=\set{X(t),t\in\T}$
be a centered Gaussian process.
Introduce the notation
\[
\rho_X(t,s)=\left(\E(X(t)-X(s))^2\right)^{1/2}, s,t\in \T.
\]
Evidently, $\rho_X$ is a pseudometric on $\T$. Also, denote
$$m(\T)=\sup_{t\in\T}\left(\E\abs{X(t)}^2\right)^{1/2}.$$
Throughout the section we assume that the following conditions hold.

\begin{enumerate}[({A}1)]
\item\label{A1}
$m(\T)<\infty$.
\item\label{A2}
The space $(\T,\rho_X)$ is separable and the process $X$ is separable on this space.
\end{enumerate}
\subsection{ Exponential maximal upper  bound for Gaussian process in terms of metric massiveness} In this subsection we present the general results concerning
exponential maximal upper  bound for Gaussian process defined on an arbitrary parameter set, in terms of metric massiveness. Let $N(u), u>0$ be the metric massiveness  of the space $(\T,\rho_X)$, that is, $N(u)$ is the number of open balls in the minimal
$u$-covering of $(\T,\rho_X)$. Consider the function  $r(x)$, $x\ge1$ satisfying the following properties:
\begin{itemize}
\item[$(i)$] $r$ is non-negative and nondecreasing;
 \item[$(ii)$]   $r(e^y), y\geq 0$ is a  convex function.
\end{itemize}
Introduce one more notation: let $I_r(x)=\int_0^{x}r(N(v))dv, x>0$.
\begin{theorem}\label{th:1}
Let $I_r(m(\T))<\infty.$ Then the following bounds hold:
\begin{itemize}
\item[$(i)$] For   any $\theta\in(0,1)$  and any  $\lambda>0$
\begin{equation}\label{eq:2}
\E\exp\set{\lambda\sup_{t\in\T}\abs{X(t)}}\le 2A_1(\lambda, \theta),
\end{equation}
where
\[
A_1(\lambda,\theta)= \exp\set{\frac{\lambda^2m^2(\T)}{2(1-\theta)^2}}
r^{(-1)}\left(\frac{I_r(\theta m(\T))}{\theta m(\T)}\right),
\]
$r^{(-1)}(t)$ is the generalized inverse function of $r(t)$ that is
\[
r^{(-1)}(t)=\sup\set{u\ge0:r(u)\le t}.
\]
\item[$(ii)$] For any $\theta\in(0,1)$ and any  $\mu>0$
\begin{equation}\label{eq:3}
\P\set{\sup_{t\in\T}\abs{X(t)}\ge \mu}\le2A_2(\mu,\theta),
\end{equation}
where
\[
A_2(\mu,\theta)=\exp\set{-\frac{\mu^2(1-\theta)^2}{2m^2(\T)}}
r^{(-1)}\left(\frac{I_r(\theta m(\T))}{\theta m(\T)}\right).
\]
\end{itemize}
\end{theorem}
\begin{proof} $(i)$ First, we simplify the notation: let  $m:=m(\T).$ Now our goal is to  establish the following bound: for arbitrary $\theta\in(0,1)$ and any sequence  $r_n>0$  such  that
$\sum_{n=1}^\infty\frac{1}{r_n}=1$,
\begin{equation}\label{eq:4}
\E\exp\set{\lambda\sup_{t\in\T}\abs{X(t)}}
\le\prod_{n=1}^\infty\left[2N( m\theta^n)
\exp\set{\tfrac12\lambda^2r_n^2\theta^{2(n-1)}m^2}\right]^{\frac{1}{r_n}}.
\end{equation}
Let $\theta \in (0, 1)$ and put $u_n =
m \theta^n,$ $n\geq 0$.
Denote by $ S_n$ a minimal
$u_n$-net in the set $\T$ with respect to the
pseudometric $\rho_X$ and put
$S = \bigcup_{n=0}^\infty S_n$.
According to condition  $(A\ref{A2})$, the set $S$ is countable and everywhere dense in $\T$ with respect to the pseudometric $\rho_X$, and the process $X$ is continuous in probability in $(\T, \rho_X)$.
Therefore the set $S$ is a $\rho_X$-separability set for the process $X$ and moreover
\[
\sup_{t \in \T} \abs{X(t)} = \sup_{t \in S} \abs{X(t)}
\]
with probability 1.

Suppose that $ t \in S$.
Then there exists a number
$ n(t)$ such that $t \in S_{n(t)}$.
Define a function $\alpha_k:S \to S_k$,
$k \geq 0$   as   $\alpha_k(x) =x$
if $x \in S_k$ and $\alpha_k (x)$ is the point of $S_k$ closest to $x$ if $x \notin S_k$.
If there is more than one closest point then we may choose
any of these points.
The family of maps
$\set{\alpha_k, k \geq 0}$
is called the $\alpha$-procedure for choosing points in $S$.
Using the $\alpha$-procedure we can choose a sequence of points $t_{n(t)} = t$, $t_{n(t)-1} = \alpha_{n(t)-1}\left(t_{n(t)}\right), \ldots , t_1 = \alpha_1 (t_2)$
such that
$t_k \in S_k$, $k =1, \ldots , n(t)$ and $\rho_X( t_k, \alpha_{k-1} (t_k))\leq u_{k-1}$.

Evidently,
\[
X(t) = X(t_1) + \sum_{k=2}^{n(t)} (X(t_k) - X(t_{k-1})).
\]
Therefore we have an upper  bound
\[
\sup_{t \in S}\abs{X(t)} \leq \max_{s \in S_1} \abs{X(s)} +
\sum_{n=2}^\infty \max_{s \in S_n}\abs{X(s) - X(\alpha_{n-1} (s))}.
\]
Take any sequence of numbers $r_n >0$, $n \geq 1$ such that
$\sum_{n=1}^\infty r_n^{-1} =1$.
It follows from the   H\"{o}lder  inequality that  for any $\lambda > 0$
\begin{multline}\label{bound11}
\E \exp \set{ \lambda \sup_{t \in S}\abs{X(t)}}
\leq\E \exp \set{ \lambda \left(\max_{s \in S_1}\abs{X(s)}+ \sum_{n=2}^\infty
\max_{s \in S_n} \abs{X(s) - X(\alpha_{n-1} (s))}\right)}\\
\leq\left[ \E \exp \set{\lambda r_1  \max_{s \in S_1} |X(s)|}
\right]^{\frac{1}{r_1}}
\prod_{n=2}^\infty \left[ \E \exp \set{ \lambda r_n \max_{s \in S_n}
\abs{X(s) - X(\alpha_{n-1} (s))}} \right]^{\frac{1}{r_n}}.
\end{multline}

Furthermore, all the multipliers in the right-hand side of \eqref{bound11}, except the 1st one,  can be estimated as
 \begin{equation}\begin{gathered}\label{bound12}
 \E \exp \set{ \lambda r_n \max_{s \in S_n}
\abs{X(s) - X(\alpha_{n-1} (s))}} \\
\leq N(u_n) \max_{s \in S_n } \E \exp \set{ \lambda r_n
\abs{X(s) - X(\alpha_{n-1}(s))}},
 \end{gathered}\end{equation}
and, in addition,
\[
\left(\E\abs{X(s) - X(\alpha_{n-1} (s))}^2\right)^{\frac{1}{2}} =
\rho_X (s, \alpha_{n-1} (s))\leq u_{n-1} = \theta^{n-1} u_0
\]
 for $s \in S_n.$ Therefore, for any $n \geq 2$
 \begin{equation}\begin{gathered}\label{bound13}
 \max_{s \in S_n}\E \exp \set{ \lambda r_n
\abs{X(s) - X(\alpha_{n-1}(s))}}
 \leq \max_{s \in S_n} \Big( \E \exp \set{ \lambda r_n
(X(s) - X(\alpha_{n-1}(s)) )} \\
+ \E \exp \set{ - \lambda r_n
(X(s) - X(\alpha_{n-1}(s)) )}\Big)\\
 =  \max_{s \in S_n } 2 \exp \set{ \frac{\lambda^2}{2} r_n^2
\E (X(s) - X(\alpha_{n-1}(s)))^2}
\leq 2 \exp \set{ \frac{\lambda^2}{2} r_n^2
 ( \theta^{n-1} u_0 )^2}.
\end{gathered}\end{equation}

Now we estimate the first multiplier in the right-hand side of \eqref{bound11}:
\begin{equation}\begin{gathered}\label{bound14}
\E \exp \set{ \lambda r_1 \max_{s \in S_1} \abs{X(s)}}
 \leq N(u_1) \max_{s \in S_1} \E \exp \set{ \lambda r_1
\abs{X(s)}}\\
 \leq 2 N(u_1) \exp \set{ \frac{\lambda^2 r_1^2}{2}
\max_{s \in S_1} \E \abs{X(s)}^2}
 \leq 2 N(u_1) \exp \set{ \frac{\lambda^2 r_1^2}{2}
u_0^2}.
\end{gathered}\end{equation}
Taking into account that it follows from the separability of $X$ that
\[
\E \exp \set{ \lambda \sup_{t \in \T} \abs{X(t)}} =
\E \exp \set{\lambda \sup_{t \in S} \abs{X(t)}},
\]
we get inequality   \eqref{eq:4} from \eqref{bound11}--\eqref{bound14}.
It follows from \eqref{eq:4} that
\begin{equation}\label{eq:5}
\E\exp\set{\lambda\sup_{t\in\T}\abs{X(t)}}
\le2\exp\set{\sum_{n=1}^\infty\frac{1}{r_n}H(m\theta^n)
+\sum_{n=1}^\infty\frac{\lambda^2m^2}{2}\theta^{2(n-1)}r_n},
\end{equation}
where
$H(u)=\log N(u)$ is the metric entropy.
Now, choose  $r_n=\frac{1}{\theta^{n-1}(1-\theta)}$.
Then
\begin{equation}\label{eq:6}
\E\exp\set{\lambda\sup_{t\in\T}\abs{X(t)}}
\le2\exp\set{(1-\theta)\sum_{n=1}^\infty\theta^{n-1}H(m\theta^n)
+\frac{\lambda^2m^2}{2(1-\theta)^2}}.
\end{equation}
Since $r(e^x)$ is a convex function, we have that
\begin{equation}\label{eq:7}
\begin{split}
r^{(-1)}&\left(r\left(\exp\set{\sum_{n=1}^\infty(1-\theta)\theta^{n-1}H(\theta^nm)}\right)\right)\\
&\le r^{(-1)}\left(\sum_{n=1}^\infty(1-\theta)\theta^{n-1}r\left(\exp\set{H(\theta^nm)}\right)\right)\\
&=r^{(-1)}\left(\sum_{n=1}^\infty(1-\theta)\theta^{n-1}r\left(N(\theta^nm)\right)\right)\\
&\le r^{(-1)}\left(\sum_{n=1}^\infty(1-\theta)\theta^{n-1}\frac{1}{m\theta^n(1-\theta)}
    \int_{\theta^{n+1}m}^{\theta^nm}r(N(u))du\right)\\
&= r^{(-1)}\left(\frac{1}{m\theta}\int_0^{m\theta}r(N(u))du\right).
\end{split}
\end{equation}
Now, inequality \eqref{eq:2}   follows from \eqref{eq:6} and \eqref{eq:7}.

$(ii)$  Now we are in position to establish inequality  \eqref{eq:3}.
Let $u>0$, $0<\theta<1$, $\lambda>0$.
Then Chebyshev's inequality and \eqref{eq:2} yield that
\begin{align*}
\P\set{\sup_{t\in\T}\abs{X(t)}\ge u}
&\le\E\exp\set{\lambda\sup_{t\in\T}\abs{X(t)}}\exp\set{-\lambda u}\\
&\le 2r^{(-1)}\left(\frac{I_r(\theta m)}{\theta m}\right)
\exp\set{\frac{\lambda^2m^2}{2(1-\theta)^2}-\lambda u}.
\end{align*}
Minimizing $\frac{\lambda^2m^2}{2(1-\theta)^2}-\lambda u$ with respect to $\lambda>0$, we note that   minimum is achieved at the point $\lambda=\frac{u(1-\theta)^2}{m^2}$,
whence \eqref{eq:3} immediately follows.
\end{proof}

Applying this result to the parameter set $\T=[a,b]$, we get the following result.

\begin{corollary}\label{cor:1}
Let $\T=[a,b], $
$X=\set{X(t),t\in[a,b]}$
be a centered separable Gaussian  process  and $m:=m([a,b])=\sup_{t\in[a,b]}\left(\E\abs{X(t)}^2\right)^{1/2}<\infty$.
Assume that there exists a   strictly increasing function
$\sigma=\set{\sigma(h), h>0}$
such that
$\sigma(h)>0$, $h>0$,
$\sigma(h)\downarrow0$ as $h\downarrow0$,
and
\[\sup_{\abs{t-s}<h}\left(\E\abs{X(t)-X(s)}^2\right)^{1/2}\le\sigma(h).\]
Then for   any $\theta\in(0,1)$ and any  $\lambda>0$
\begin{equation}\label{bound17}
\E\exp\set{\lambda\sup_{t\in[a,b]}\abs{X(t)}}\le 2A_3(\lambda,\theta),
\end{equation}
where
\begin{equation}\label{eq:a1}
A_3(\lambda,\theta)=\exp\set{\frac{\lambda^2m^2}{2(1-\theta)^2}}
r^{(-1)}\left(\frac{\widetilde I_r(\theta m)}{\theta m}\right),
\end{equation}
and \[
\widetilde I_r(x)=\int_0^{x}
r\left(\frac{b-a}{2\sigma^{(-1)}(v)}+1\right)dv.
\]
Indeed, in this case condition $(A\ref{A2})$ holds and
$N(v)\le\frac{b-a}{2\sigma^{(-1)}(v)}+1$, whence \eqref{bound17}--\eqref{eq:a1} immediately follow.
\end{corollary}

\begin{corollary}\label{cor:2}
Let we can put
$\sigma(h)=ch^\beta$ with  $c>0$, $0<\beta\le1$   in Corollary~\ref{cor:1}.
Then for    any $\theta\in(0,1)$ and any  $\lambda>0$
\begin{equation}\label{eq:a2}
\E\exp\set{\lambda\sup_{t\in[a,b]}\abs{X(t)}}
\le 2^{\frac2\beta-1}\exp\set{\frac{\lambda^2m^2}{2(1-\theta)^2}}
\left(\frac{2^{2/\beta-1}(b-a)c^{1/\beta}}{(\theta m)^{1/\beta}}+1\right).
\end{equation}

Indeed, consider
$r(x)=x^{\alpha}-1$, $x\ge1$,
where $0<\alpha<\beta$.
Since
$\sigma^{(-1)}(s)=\left(\frac{s}{c}\right)^{1/\beta}$,
we have
\begin{align*}
\widetilde I_r(\theta m)
&=\int_0^{\theta m}\left(\left(\frac{b-a}{2\sigma^{(-1)}(s)}+1\right)^\alpha-1\right) ds
\le\int_0^{\theta m}\left(\frac{b-a}{2\sigma^{(-1)}(s)}\right)^\alpha ds\\
&=\int_0^{\theta m}\left(\frac{(b-a)c^{1/\beta}}{2s^{1/\beta}}\right)^\alpha ds
=\left(\frac{(b-a)c^{1/\beta}}{2}\right)^\alpha
   \frac{(\theta m)^{1-\alpha/\beta}}{1-\frac\alpha\beta}.
\end{align*}
Therefore in this case
\[
A_3(\lambda,\theta)\le\exp\set{\frac{\lambda^2m^2}{2(1-\theta)^2}}
\left(\left(\frac{(b-a)c^{1/\beta}}{2}\right)^\alpha\frac{1}{\left(1-\frac\alpha\beta\right)(\theta m)^{\alpha/\beta}}+1\right)^{\frac1\alpha}.
\]
Applying the elementary inequality
$(a+b)^p\le2^{p-1}(a^p+b^p)$, $p\ge1$,
we get
\begin{equation}\label{eq:a3}
A_3(\lambda,\theta)\le\exp\set{\frac{\lambda^2m^2}{2(1-\theta)^2}}
\cdot2^{\frac1\alpha-1}\left(\frac{(b-a)c^{1/\beta}}{2\left(1-\frac\alpha\beta\right)^{1/\alpha}(\theta m)^{1/\beta}}+1\right).
\end{equation}
Now \eqref{eq:a2} follows from \eqref{eq:a3} if we put $\alpha=\frac\beta2$.
 \end{corollary}

 \subsection{ Exponential maximal upper  bound for the weighted Gaussian process defined on the half-axis}
 Now, let
$X=\set{X(t),t\ge0}$
be a centered Gaussian process and
$a(t)>0$ be a continuous strictly increasing function such that
$a(t)\to\infty$ as $t\to\infty$.
Introduce the sequence $b_0=0$, $b_{k+1}>b_k$, $b_k\to\infty$ as $k\to\infty$. Denote
$a_k=a(b_k)$ and $m_k=m([b_k,b_{k+1}])=\sup_{t\in[b_k,b_{k+1}]}\left(\E\abs{X(t)}^2\right)^{1/2}.$ Our goal is to get exponential maximal upper  bound for the weighted Gaussian process $\frac{ {X(t)}}{a(t)}$, applying the above results, in particular, Corollary \ref{cor:2}.
\begin{theorem}\label{th:2}
Let the following conditions hold:
\begin{itemize}

\item [$(i)$]
There exist $c_k>0$ and $0<\beta<1$ such that
\begin{gather*}
\sup_{\substack{t,s\in[b_k,b_{k+1}]\\\abs{t-s}\le h}}
\left(\E\abs{X(t)-X(s)}^2\right)^{1/2}\le c_kh^\beta;
\end{gather*}

\item [$(ii)$]\begin{equation}\label{eq:sum1}
0<m_k<\infty \;\textit{and}\; A=\sum_{k=0}^\infty\frac{m_k}{a_k}<\infty;
\end{equation}
\item [$(iii)$] There exists $0<\gamma\le1$ such that
\begin{equation}\label{eq:sum2}
\sum_{k=0}^\infty\frac{ m_k^{1-\gamma/\beta}(b_{k+1}-b_k)^\gamma c_k^{\gamma/\beta}}{a_k}<\infty.
\end{equation}
\end{itemize}
Then for  any $\theta\in(0,1)$ and any  $\lambda>0$
\[
I(\lambda)=\E\exp\set{\lambda\sup_{t>0}\frac{\abs{X(t)}}{a(t)}}
\le2^{\frac2\beta-1}\exp\set{\frac{\lambda^2A^2}{2(1-\theta)^2}} A_4(\theta,\gamma),
\]
where
\[
A_4(\theta,\gamma)=
\exp\set{\frac{1}{\gamma A}\left(\sum_{k=0}^\infty\frac{m_k^{1-\gamma/\beta}}{a_k}(b_{k+1}-b_k)^\gamma c_k^{\gamma/\beta}\right)\left(\frac{2^{2/\beta-1}}{\theta^{1/\beta}}\right)^\gamma}.
\]
\end{theorem}

\begin{proof}
Let $r_k>0$, $k=0,1,2\ldots$ and
$\sum_{k=0}^\infty\frac{1}{r_k}=1$.
Then for  any   $\lambda>0$
\[
I(\lambda)
\le\E\exp\set{\lambda\sum_{k=0}^\infty\sup_{t\in[b_k,b_{k+1}]}\frac{\abs{X(t)}}{a(t)}}
\le\prod_{k=0}^\infty\left(\E\exp\set{\lambda r_k\sup_{t\in[b_k,b_{k+1}]}\frac{\abs{X(t)}}{a(t)}}\right)^{\frac{1}{r_k}}.
\]
It follows from Corollary~\ref{cor:2} that for any $\theta\in(0,1)$
\begin{multline*}
\E\exp\set{\lambda r_k\sup_{t\in[b_k,b_{k+1}]}\frac{\abs{X(t)}}{a(t)}}
\le\E\exp\set{\lambda r_k\sup_{t\in[b_k,b_{k+1}]}\frac{\abs{X(t)}}{a_k}}\\
\le2^{\frac2\beta-1}\exp\set{\frac{\lambda^2r_k^2m_k^2}{2(1-\theta)^2a_k^2}}
\left(1+\frac{b_{k+1}-b_k}{\theta^{1/\beta}}2^{2/\beta-1}
\left(\frac{c_k}{m_k}\right)^{1/\beta}\right).
\end{multline*}
Therefore,
\begin{align*}
I(\lambda)&\le2^{\frac2\beta-1}\exp\set{\frac{\lambda^2}{2(1-\theta)^2}\sum_{k=0}^\infty\frac{r_km_k^2}{a_k^2}}
\prod_{k=0}^\infty\left[1+\frac{b_{k+1}-b_k}{\theta^{1/\beta}}2^{2/\beta-1}
\left(\frac{c_k}{m_k}\right)^{1/\beta}\right]^\frac{1}{r_k}\\
&=2^{\frac2\beta-1}\exp\set{\frac{\lambda^2}{2(1-\theta)^2}\sum_{k=0}^\infty\frac{r_km_k^2}{a_k^2}}\\
&\quad\times
\exp\set{\sum_{k=0}^\infty\frac{1}{r_k}\log\left[1+\frac{b_{k+1}-b_k}{\theta^{1/\beta}}2^{2/\beta-1}
\left(\frac{c_k}{m_k}\right)^{1/\beta}\right]}.
\end{align*}
Recall the elementary inequality: for $0<\gamma\le1$ and $x\ge0$,
\begin{equation}\label{eq:log}
\log(1+x)=\frac1\gamma\log(1+x)^\gamma
\le\frac{x^\gamma}{\gamma}.
\end{equation}
Taking this into account, we continue with the upper bound for $I(\lambda)$:
\begin{multline*}
I(\lambda)\le2^{\frac2\beta-1}\exp\set{\frac{\lambda^2}{2(1-\theta)^2}\sum_{k=0}^\infty\frac{r_km_k^2}{a_k^2}}\\
\times\exp\set{\frac1\gamma\sum_{k=0}^\infty\frac{1}{r_k}\left(\frac{b_{k+1}-b_k}{\theta^{1/\beta}}2^{2/\beta-1}
\left(\frac{c_k}{m_k}\right)^{1/\beta}\right)^\gamma}.
\end{multline*}
Let $r_k=\frac{A a_k  }{m_k}$.
Then we get immediately the claimed upper bound:
\begin{multline*}
I(\lambda)\le2^{\frac2\beta-1}\exp\set{\frac{\lambda^2A^2}{2(1-\theta)^2}}
\exp\set{\frac1{\gamma A}\sum_{k=0}^\infty \frac{m_k}{a_k}\left(\frac{b_{k+1}-b_k}{\theta^{1/\beta}}2^{2/\beta-1}
\left(\frac{c_k}{m_k}\right)^{1/\beta}\right)^\gamma}\\
=2^{\frac2\beta-1}\exp\set{\frac{\lambda^2A^2}{2(1-\theta)^2}}
\exp\set{\frac1{\gamma A}\left(\sum_{k=0}^\infty \frac{m_k^{1-\gamma/\beta}}{a_k}(b_{k+1}-b_k)^\gamma c_k^{\gamma/\beta}\right)
\left(\frac{2^{2/\beta-1}}{\theta^{1/\beta}}\right)^\gamma}.
\qedhere
\end{multline*}
\end{proof}

\begin{corollary}\label{cor:3}
Let the assumptions of Theorem~\ref{th:2} hold.
Then for any    $\theta\in(0,1)$ and any  $u>0$  the following inequality holds:
\begin{equation}\label{eq:a4}
\P\set{\sup_{t>0}\frac{\abs{X(t)}}{a(t)}>u}
\le2^{\frac2\beta-1}\exp\set{-\frac{u^2(1-\theta)^2}{2A^2}}A_4(\theta,\gamma).
\end{equation}
Indeed, from Chebyshev's inequality we get that
\begin{equation}\label{eq:a5}
\begin{split}
\P\set{\sup_{t>0}\frac{\abs{X(t)}}{a(t)}>u}
&\le\frac{\E\exp\set{\lambda\sup_{t>0}\frac{\abs{X(t)}}{a(t)}}}{\exp\set{\lambda u}}\\
&\le2^{\frac2\beta-1}\exp\set{\frac{\lambda^2A^2}{2(1-\theta)^2}-\lambda u}A_4(\theta,\gamma).
\end{split}
\end{equation}
The inequality \eqref{eq:a4} follows from \eqref{eq:a5} if we put  $\lambda=\frac{u(1-\theta)^2}{A^2}$.
 \end{corollary}

\begin{corollary}\label{cor:4}
Let the assumptions of Theorem~\ref{th:2} hold.
Then for any $u>A$ we can get the following bound:
\begin{equation}\label{eq:a6}
\P\set{\sup_{t>0}\frac{\abs{X(t)}}{a(t)}>u}
\le2^{\frac2\beta-1}\sqrt{e}\exp\set{-\frac{u^2}{2A^2}}
A_4\left(1-\sqrt{1-\frac{A^2}{u^2}},\gamma\right).
\end{equation}
 Indeed, the inequality \eqref{eq:a6} follows from \eqref{eq:a4} if we put $\theta=1-\sqrt{1-\frac{A^2}{u^2}}$.
\end{corollary}
\begin{corollary}\label{cor:5}
Let the assumptions of Theorem~\ref{th:2} hold.
Then for all $t>0$ we have with probability 1 the following bound,
\[
\abs{X(t)}\le a(t)\xi,
\]
where $\xi$ is   non-negative   random variable whose   distribution has  the tail admitting the following   upper bound:   for any $u>A$
\[
\P\set{\xi>u}\le 2^{\frac2\beta-1}\sqrt{e}\exp\set{-\frac{u^2}{2A^2}}
A_4\left(1-\sqrt{1-\frac{A^2}{u^2}},\gamma\right).
\]
\end{corollary}
\subsection{ Exponential maximal upper  bound for Gaussian process in the bounded strip on the plane}
Additionally, we need in exponential maximal upper  bound for Gaussian process defined in the bounded strip on the plane. We get it, applying Theorem \ref{th:1}.  So, let $0\le a<b<\infty$, $\Delta>0$,
\begin{gather*}
\T_{a,b,\Delta}=\set{\ve{t}=(t_1,t_2)\in\R_+^2:a\le t_1\le b,t_1-\Delta\le t_2\le t_1},\\ d(\ve{t},\ve{s})=\max\set{\abs{t_1-s_1},\abs{t_2-s_2}}\;\text{for}\;  \ve{t},\ve{s}\in\T_{a,b,\Delta}.
\end{gather*}
Let $r(x)$, $x\ge1$, be a non-negative nondecreasing function such that
$r\left(e^y\right)$, $y\ge0$ is a convex function.

\begin{theorem}\label{th:3}

Assume that
$X=\set{X(\ve t),\ve t\in\T_{a,b,\Delta}}$
is a centered separable Gaussian process satisfying following conditions:

\begin{itemize}

\item [$(iv)$]\begin{gather*}
m(\T_{a,b,\Delta})=\sup_{\ve t\in\T_{a,b,\Delta}}\left(\E(X(\ve t))^2\right)^{1/2}<\infty;\end{gather*}
\item [$(v)$]\begin{gather*}\sup_{\substack{d(\ve t,\ve s)\le h\\\ve t,\ve s\in\T_{a,b,\Delta}}}
\left(\E(X(\ve t)-X(\ve s))^2\right)^{1/2}\le\sigma(h),
\end{gather*}
where
$\sigma=\set{\sigma(h),h>0}$
is an increasing continuous function,
$\sigma(h)\ge0$, $\sigma(0)=0$.
\end{itemize}

If
\[
\widehat I_r(m(\T_{a,b,\Delta}))=\int_{0}^{m(\T_{a,b,\Delta})} r\left(\frac{(b-a)\Delta}{4}\left(\frac{1}{\sigma^{(-1)}(v)}+d\right)^2\right)dv<\infty,
\]
where
$d=\max\set{\frac{2}{b-a},\frac4\Delta}$,
then for    any    $\theta\in(0,1)$ and any  $\lambda>0$
\[
\E\exp\set{\lambda\sup_{\ve t\in\T_{a,b,\Delta}}\abs{X(\ve t)}}
\le2  A_5(\lambda,\theta),
\]
where
\[
 A_5(\lambda,\theta)=\exp\set{\frac{\lambda^2m^2(\T_{a,b,\Delta})}{2(1-\theta)^2}}
r^{(-1)}\left(\frac{\widehat I_r(\theta m(\T_{a,b,\Delta}))}{\theta m(\T_{a,b,\Delta})}\right).
\]
\end{theorem}

\begin{proof}
The statement follows from Theorem~\ref{th:1}, since in this case
\begin{align*}
N(v)&\le\left(\frac{b-a}{2\sigma^{(-1)}(v)}+1\right)\left(\frac{\Delta+2\sigma^{(-1)}(v)}{2\sigma^{(-1)}(v)}+1\right)\\
&=\left(\frac{b-a}{2\sigma^{(-1)}(v)}+1\right)\left(\frac{\Delta}{2\sigma^{(-1)}(v)}+2\right)\\
&=\frac{b-a}{2}\cdot\frac{\Delta}{2}\left(\frac{1}{\sigma^{(-1)}(v)}+\frac{2}{b-a}\right)\left(\frac{1}{\sigma^{(-1)}(v)}+\frac{4}{\Delta}\right)\\
&\le\frac{(b-a)\Delta}{4}\left(\frac{1}{\sigma^{(-1)}(v)}+d\right)^2.
\qedhere
\end{align*}
\end{proof}

\begin{corollary}\label{cor:b1}
Let in Theorem~\ref{th:3} $\sigma(h)=ch^\beta$ for some $c>0$ and $\beta\in(0,1]$, $m=m(\T_{a,b,\Delta})$.
Then for all $\lambda>0$, $0<\eps<\beta$, $0<\theta<1$
\begin{multline*}
\E\exp\set{\lambda\sup_{\ve t\in\T_{a,b,\Delta}}\abs{X(\ve t)}}
\le2^{\frac2\eps-2}(b-a)\Delta\exp\set{\frac{\lambda^2m^2}{2(1-\theta)^2}}\\
\times\left(\frac{c^{2/\beta}}{\left(1-\frac{\eps}{\beta}\right)^{2/\eps}(\theta m)^{2/\beta}}+d^2\right),
\end{multline*}
where
$d=\max\set{\frac{2}{b-a},\frac4\Delta}$.
Indeed, we can choose
\[
r(s)=
\begin{cases}
0, & 1\le s\le\frac{d^2(b-a)\Delta}{4},\\
s^{\eps/2}-\left(\frac{d^2(b-a)\Delta}{4}\right)^{\eps/2}, & s>\frac{d^2(b-a)\Delta}{4}.
\end{cases}
\]
Then
$r^{(-1)}(t)=\left(t+\left(\frac{d^2(b-a)\Delta}{4}\right)^{\eps/2}\right)^{\frac2\eps}$, and
\begin{align*}
\widehat I_r(\theta m)
&=\int_{0}^{\theta m} r\left(\frac{(b-a)\Delta}{4}\left(\frac{c^{1/\beta}}{v^{1/\beta}}+d\right)^2\right)dv\\
&=\int_{0}^{\theta m} \left(\left(\frac{(b-a)\Delta}{4}\left(\frac{c^{1/\beta}}{v^{1/\beta}}+d\right)^2\right)^{\eps/2}-\left(\frac{d^2(b-a)\Delta}{4}\right)^{\eps/2}\right) dv\\
&=\left(\frac{(b-a)\Delta}{4}\right)^{\frac\eps2}\int_{0}^{\theta m} \left(\left(\frac{c^{1/\beta}}{v^{1/\beta}}+d\right)^{\eps}-d^\eps\right) dv\\
&\le\left(\frac{(b-a)\Delta}{4}\right)^{\frac\eps2}\int_{0}^{\theta m}\frac{c^{\eps/\beta}}{v^{\eps/\beta}} dv
=\left(\frac{(b-a)\Delta}{4}\right)^{\frac\eps2}\frac{c^{\eps/\beta}(\theta m)^{1-\eps/\beta}}{1-\frac{\eps}{\beta}}.
\end{align*}
Hence,
\begin{align*}
A_5(\lambda,\theta)
&=\exp\set{\frac{\lambda^2m^2}{2(1-\theta)^2}}
\left(\frac{\widehat I_r(\theta m)}{\theta m}+\left(\frac{d^2(b-a)\Delta}{4}\right)^{\frac\eps2}\right)^{\frac2\eps}\\
&\le\exp\set{\frac{\lambda^2m^2}{2(1-\theta)^2}}\frac{(b-a)\Delta}{4}
\left(\frac{c^{\eps/\beta}}{\left(1-\frac\eps\beta\right)(\theta m)^{\eps/\beta}}+d^\eps\right)^{\frac2\eps}\\
&\le\exp\set{\frac{\lambda^2m^2}{2(1-\theta)^2}}\frac{(b-a)\Delta}{4}
2^{\frac2\eps-1}\left(\frac{c^{2/\beta}}{\left(1-\frac\eps\beta\right)^{2/\eps}(\theta m)^{2/\beta}}+d^2\right).
\end{align*}
\end{corollary}

Now we get the upper bound for the weighted Gaussian process defined on the bounded strip on the plane, similarly to getting Theorem \ref{th:2} from Theorem \ref{th:1}.  Let
$\T_\Delta=\set{\ve{t}=(t_1,t_2)\in\R_+^2:t_1-\Delta\le t_2\le t_1}$, $\Delta>0$,
$$d(\ve t,\ve s) =  \max\set{|t_1 - s_1|, |t_2 -s_2|},\quad\ve s, \ve t \in \T_\Delta.$$
 Also, let  $b_l$ be an increasing sequence such that $ b_0 =0, b_{l+1} -b_l \geq 1$, $b_l \to \infty, l \to \infty$,
 and $a(t) >0$ is a continuous increasing function and denote  $a_l = a(b_l)$,
\begin{gather*}
{\T}_{b_l, b_{l+1},\Delta}  = \set{\ve{t}=(t_1,t_2)\in\R_+^2:b_l\le t_1\le b_{l+1},t_1-\Delta\le t_2\le t_1}.\\
\end{gather*}

\begin{theorem}\label{th:4}
Let $0<\Delta\le2(b_{l+1}-b_l)$, $l\ge0$, $ X=\{ X (\ve t), \ve t\in\T_\Delta\}$ be  a centered Gaussian    process satisfying following conditions:

\begin{itemize}

\item [$(vi)$]\label{(vi)} $\displaystyle m_l =m(\T_{b_l, b_{l+1},\Delta})= \sup_{ \ve{t} \in {\T}_{b_l, b_{l+1},\Delta}} \left(\E (X (\ve{t}) )^2\right)^{\frac{1}{2}}<\infty;$
\item [$(vii)$]\label{(vii)}  There exist $\beta\in(0,1]$ and constants $c_l>0$ such that
\[\sup_{\substack{ d(\ve{t}, \ve{s}) \leq h,\\ \ve{t}, \ve{s} \in {\T}_{b_l, b_{l+1},\Delta}}}
\left(\E ( X (\ve{t}) - X(\ve{s}))^2 \right)^{\frac{1}{2}} \leq c_l h^\beta.
\]
\item [$(viii)$]\label{(viii)}
$\displaystyle
A = \sum_{l=0}^\infty \frac{m_l}{a_l} < \infty, \quad
\sum_{l=0}^\infty\frac{m_l\log(b_{l+1}-b_l)}{a_l} <\infty$,
and for some $ \gamma\in(0,  1]$
$\displaystyle \sum_{l=0}^\infty\frac{m_l^{1-2\gamma/\beta}c_l^{2\gamma/\beta}}{a_l}<\infty$.
\end{itemize}
Then for  any    $\theta\in(0,1)$, $\eps\in(0,\beta)$ and $\lambda>0$
\begin{align*}
I(\lambda) &= \E \exp\set{ \lambda \; \sup_{ \ve{t} \in \T_\Delta}
\frac {\abs{X(\ve{t})}}{a(t_1)}}
 \leq\exp \set{ \frac{\lambda^2A^2}{2(1- \theta)^2}}
A_6(\theta,\gamma,\eps),
\end{align*}
where
\begin{multline*}
A_6(\theta,\gamma,\eps) = \frac{2^{\frac2\eps+2}}{\Delta}
\exp\set{\frac1A\sum_{l=0}^\infty\frac{m_l\log(b_{l+1}-b_l)}{a_l}}\\
\times\exp\set{\frac{\Delta^{2\gamma}}{\gamma A\left(1-\frac{\eps}{\beta}\right)^{2\gamma/\eps}\theta^{2\gamma/\beta}4^{2\gamma}}
\sum_{l=0}^\infty\frac{c_l^{2\gamma/\beta}m_l^{1-2\gamma/\beta}}{a_l}}.
\end{multline*}
\end{theorem}

\begin{proof}
The theorem follows from Corollary \ref{cor:b1}.
Indeed, let $r_l >0$,
$\sum_{l=0}^\infty \frac{1}{r_l} =1$.
Then we easily get the following upper bounds
\begin{align*}
I(\lambda) &\leq \E \exp\set{ \lambda \; \sum_{l=0}^\infty \frac{1}{a_l} \;
\sup_{ \ve{t} \in {\T}_{b_l, b_{l+1},\Delta}} | X(\ve{t})|}\\
&\leq\prod_{l=0}^\infty \left(\E \exp \set{ \lambda \frac{r_l}{a_l} \;
\sup_{\ve{t} \in {\T}_{b_l, b_{l+1},\Delta}}\abs{X(\ve{t})}} \right)^{1/r_l}\\
&\leq
2 \prod_{l=0}^\infty \exp \set{
\frac{\lambda^2 m_l^2 r_l }{2 a_l^2 (1 - \theta)^2}}
 \left((b_{l+1}-b_l)\Delta2^{\frac2\eps-3}
 \left(\frac{c_l^{2/\beta}}{\left(1-\frac{\eps}{\beta}\right)^{2/\eps}(\theta m_l)^{2/\beta}}+\left(\frac4\Delta\right)^2\right)\right)^{1/r_l}
\\
&=2^{\frac2\eps} \prod_{l=0}^\infty \exp \set{
\frac{\lambda^2 m_l^2 r_l }{2 a_l^2 (1 - \theta)^2}}
(b_{l+1}-b_l)^{1/r_l}\left(\frac{4}{\Delta}\right)^{1/r_l}\\
&\quad\times
 \left(\frac{c_l^{2/\beta}}{\left(1-\frac{\eps}{\beta}\right)^{2/\eps}(\theta m_l)^{2/\beta}}\left(\frac{\Delta}{4}\right)^2+1\right)^{1/r_l}
\\
&=\frac{2^{\frac2\eps+2}}{\Delta}\exp \set{ \frac{\lambda^2}{2(1- \theta)^2}
 \sum_{l=0}^\infty \frac{r_l m_l^2}{a_l^2}}
\exp\set{ \sum_{l=0}^\infty\frac{\log(b_{l+1}-b_l)}{r_l}}\\
&\quad\times
\exp\set{\sum_{l=0}^\infty\frac{1}{r_l}\log\left(\frac{c_l^{2/\beta}\Delta^2}{\left(1-\frac{\eps}{\beta}\right)^{2/\eps}(\theta m_l)^{2/\beta}4^2}+1\right)}.
\end{align*}
Applying \eqref{eq:log}, we get
\begin{align*}
I(\lambda) &\leq
\frac{2^{\frac2\eps+2}}{\Delta}\exp \set{ \frac{\lambda^2}{2(1- \theta)^2}
 \sum_{l=0}^\infty \frac{r_l m_l^2}{a_l^2}}
\exp\set{ \sum_{l=0}^\infty\frac{\log(b_{l+1}-b_l)}{r_l}}\\
&\quad\times
\exp\set{\sum_{l=0}^\infty\frac{c_l^{2\gamma/\beta}\Delta^{2\gamma}}{\gamma r_l\left(1-\frac{\eps}{\beta}\right)^{2\gamma/\eps}(\theta m_l)^{2\gamma/\beta}4^{2\gamma}}}.
\end{align*}

Now we choose  $ r_l = \frac{A a_l}{m_l}$.
Then
\begin{align*}
I(\lambda) &\leq
\frac{2^{\frac2\eps+2}}{\Delta}\exp \set{ \frac{\lambda^2A^2}{2(1- \theta)^2}}
\exp\set{\frac1A\sum_{l=0}^\infty\frac{m_l\log(b_{l+1}-b_l)}{a_l}}\\
&\quad\times
\exp\set{\frac{\Delta^{2\gamma}}{\gamma A\left(1-\frac{\eps}{\beta}\right)^{2\gamma/\eps}\theta^{2\gamma/\beta}4^{2\gamma}}
\sum_{l=0}^\infty\frac{c_l^{2\gamma/\beta}m_l^{1-2\gamma/\beta}}{a_l}}.
 \qedhere
\end{align*}
\end{proof}

\begin{corollary}\label{cor:aI}
Let the assumptions of Theorem~\ref{th:4} hold.
Then for  all   $\theta\in(0,1)$, $\eps\in(0,\beta)$ and  $u>0$,
\[
\P \set{ \sup_{\ve{t} \in \T_\Delta} \frac{\abs{X(\ve{t})}}{a(t_1)}>u}
\leq \exp \set{ - \frac{u^2 (1- \theta)^2}{2 A^2}}A_6(\theta,\gamma,\eps).
\]
\end{corollary}

\begin{corollary}\label{cor:aII}
Let the assumptions of Theorem~\ref{th:4} hold.
Then for any $\eps\in(0,\beta)$ and
$u > A$  we have that
\[
\P\set{\sup_{\ve{t} \in \T_\Delta} \frac{|X(\ve{t})|}{a(t_1)}>u}
\le \sqrt{e} \exp\set{-\frac{u^2}{2 A^2}}
A_6\left( 1- \sqrt{1- \frac{A^2}{u^2}},\gamma,\eps\right).
\]
\end{corollary}

\begin{corollary}\label{cor:aIII}
Let the assumptions of Theorem~\ref{th:4} hold.
Then  for all $\ve t\in\T_\Delta$
\[
\abs{X (\ve{t})}\leq a( t_1 ) \xi \quad\text{a.\,s.},
\]
where $\xi$ is such non-negative random variable  that
\[
\P \set{ \xi > u} \leq \exp\set{- \frac{u^2 (1- \theta)^2}{2A^2}}  A_6(\theta,\gamma,\eps),
\]
and for $u > a$
\[
\P \set{ \xi > u} \leq \sqrt{e} \exp
\set{ -\frac{u^2}{2 A^2}}
A_6\left( 1- \sqrt{1- \frac{A^2}{u^2}},\gamma,\eps\right).
\]
\end{corollary}

\section{Asymptotic growth with probability 1 of  multifractional Brownian motion}\label{sec:mBm}
In this section we apply the results of Section  \ref{sec:Gauss} to get the asymptotic growth of the trajectories of multifractional Brownian motion.
\subsection{Definition and assumptions}
Let $H\colon\R_+\to(0,1)$ be a continuous function.

The (harmonizable) multifractional Brownian motion (mBm) with functional parameter $H$ was introduced in \cite{BenassiJaffardRoux97}.
It is defined by
\[
Y(t)=\int_\R\frac{e^{itu}-1}{\abs{u}^{H_t+1/2}}\Wt(du),
\quad t\ge0,
\]
where $\Wt(du)$ is the ``Fourier transform'' of the white noise $W(du)$, that is  a unique complex-valued random measure such that for all $f\in L^2(\R)$
\[
\int_\R f(u)W(du)=\int_\R \widehat f(u)\Wt(du)\quad\text{a.\,s.},
\]
see \cite{BenassiJaffardRoux97,StoevTaqqu}.

In what follows we assume that the function $H$ satisfies the following conditions:
\begin{enumerate}[(H1)]
\item\label{H1}
There exist constants $0<h_1<h_2<1$ such that for any $t\ge0$
\[
h_1\le H_t\le h_2.
\]
\item\label{H3}
There exist constants $D>0$ and $\kappa\in(0,1]$ such that for all $t\ge s>0$
\[
\abs{H_t-H_s}\le D\abs{t-s}^\kappa.
\]
\end{enumerate}

It is known~\cite{Ayache-covar}  that
\begin{equation}\label{eq:var-mBm}
\left(\E\abs{Y(t)}^2\right)^{1/2}=C(H_t)t^{H_t},
\end{equation}
where
$C(H)=\left(\frac{\pi}{H\Gamma(2H)\sin(\pi H)}\right)^{1/2}$.
Since the function $C(H)$ is bounded on $[h_1,h_2]$, we have under assumptions (H\ref{H1})--(H\ref{H3})
\begin{equation}\label{eq:variance}
\left(\E\abs{Y(t)}^2\right)^{1/2}\le K_1 t^{h_2},
\quad t\ge1,
\end{equation}
for some $K_1>0$.

\subsection{Upper bounds  for the incremental variances of mBm}
The first result gives the upper bound for the distance variance function for multifractional Brownian motion.
\begin{lemma}\label{l:1}
Under the assumption (H\ref{H1}),
there exists a constant $K_2>0$ such that for all $t\ge s\ge0$
\begin{equation} \label{bound1}
\E(Y(t)-Y(s))^2
\le K_2\abs{t-s}^{2H_t}+K_2(H_t-H_s)^2z^2(s),
\end{equation}
where
\[
z(s)=
\begin{cases}
s^{h_2}\left(\log^2s+1\right)^{1/2}, & s\ge1,\\
1, & 0<s<1;
\end{cases}
\]
\end{lemma}

\begin{proof}
By the isometry property,
\begin{align*}
\E(Y(t)-Y(s))^2
&=\E\left(\int_\R\left(\frac{e^{itu}-1}{\abs{u}^{H_t+1/2}}
-\frac{e^{isu}-1}{\abs{u}^{H_s+1/2}}\right)\Wt(du)\right)^2\\
&=\int_\R\abs{\frac{e^{itu}-1}{\abs{u}^{H_t+1/2}}
-\frac{e^{isu}-1}{\abs{u}^{H_s+1/2}}}^2du\\
&=\int_\R\abs{\frac{e^{itu}-1}{\abs{u}^{H_t+1/2}}
-\frac{e^{isu}-1}{\abs{u}^{H_t+1/2}}
+\frac{e^{isu}-1}{\abs{u}^{H_t+1/2}}
-\frac{e^{isu}-1}{\abs{u}^{H_s+1/2}}}^2du\\
&\le2(I_1+I_2),
\end{align*}
where
\begin{align*}
I_1&=\int_\R\frac{\abs{e^{itu}-e^{isu}}^2}{\abs{u}^{2H_t+1}}du,\\
I_2&=\int_\R\abs{e^{isu}-1}^2\left(\abs{u}^{-H_t-1/2}-\abs{u}^{-H_s-1/2}\right)^2du.
\end{align*}
Consider now $I_1$ and $I_2$ separately. For $I_1$ we get the following bound
\[
I_1=\int_\R\frac{\abs{e^{i(t-s)u}-1}^2}{\abs{u}^{2H_t+1}}du
=\abs{t-s}^{2H_t}\int_\R\frac{\abs{e^{iv}-1}^2}{\abs{v}^{2H_t+1}}dv
=\abs{t-s}^{2H_t}\int_\R\frac{4\sin^2\frac v2}{\abs{v}^{2H_t+1}}dv.
\]
Hence,
$I_1\le C_1\abs{t-s}^{2H_t}$,
where
\begin{align*}
C_1&=\int_{\abs{v}<1}\frac{4\sin^2\frac v2}{\abs{v}^{2h_2  +1}}dv
+\int_{\abs{v}>1}\frac{4\sin^2\frac v2}{\abs{v}^{2h_1+1}}dv\\
&\le\int_{\abs{v}<1}\frac{1}{\abs{v}^{2h_2-1}}dv
+\int_{\abs{v}>1}\frac{4}{\abs{v}^{2h_1+1}}dv
<\infty.
\end{align*}

Consider $I_2$. By the mean value theorem,
\[
\abs{u}^{-H_t-1/2}-\abs{u}^{-H_s-1/2}
=-\abs{u}^{-h(u)-1/2}\log\abs{u}(H_t-H_s),
\]
where $h(u)\in[h_1,h_2]$.
Therefore,
\begin{align*}
I_2&=(H_t-H_s)^2\int_\R\abs{e^{isu}-1}^2\abs{u}^{-2h(u)-1}\log^2\abs{u}du\\
&=(H_t-H_s)^2\int_\R\frac{\abs{e^{iv}-1}^2}{\abs{v}^{2h(v/s)+1}}
s^{2h(v/s)}(\log\abs{v}-\log s)^2dv\\
&\le
\begin{cases}
s^{2h_1}(H_t-H_s)^2I_3,  & 0\le s<1,\\
s^{2h_2}(H_t-H_s)^2I_3, & s\ge1,
\end{cases}
\end{align*}
where
\begin{align*}
I_3&=\int_\R\frac{\abs{e^{iv}-1}^2}{\abs{v}^{2h(v/s)+1}}
(\log\abs{v}-\log s)^2dv\\
&\le2\left(\int_\R\frac{4\sin^2\frac v2}{\abs{v}^{2h(v/s)+1}}
\log^2\abs{v}dv+\log^2s\int_\R\frac{4\sin^2\frac v2}{\abs{v}^{2h(v/s)+1}}dv\right)\\
&\le2\left(C_2+C_1\log^2s\right),\\
C_2&=\int_{\abs{v}<1}\frac{4\sin^2\frac v2}{\abs{v}^{2h_2+1}}\log^2\abs{v}dv
+\int_{\abs{v}>1}\frac{4\sin^2\frac v2}{\abs{v}^{2h_1+1}}\log^2\abs{v}dv
<\infty.
\end{align*}
Note that $f(s)=s^{2h_1}\left(C_2+C_1\log^2s\right)$, $s>0$,
is a continuous function and $f(s)\to0$ as $s\downarrow0$.
Therefore, $f$ is bounded on $[0,1]$, whence
$I_2\le C_3(H_t-H_s)^2z^2(s)$
for some $C_3>0$.
\end{proof}

\begin{rem}\label{remark1}
\begin{itemize}
\item[$(a)$] Denote $h_3=\min\set{h_1,\kappa}$. It immediately follows from \eqref{bound1}, from the fact that multifractional process is Gaussian, and from the Kolmogorov theorem that under conditions $(H1)$ and $(H2)$ process $Y$ with probability 1 has H\"{o}lder trajectories up to order $h_3$  on any finite interval.
    \item[$(b)$] Bound \eqref{bound1} is inconvenient in the sense that it contains two different exponents of $|t-s|$ and therefore one should every time relate corresponding terms  depending upon the value of $|t-s|$. To avoid this technical difficulty, we establish   the next result. Denote   $h_4=\max\set{h_2,\kappa}$, $h_5=h_4-h_3$.
    \end{itemize}
\end{rem}

\begin{lemma}\label{l:2}
Assume that the Hurst function $H$ satisfies the conditions (H\ref{H1}) and (H\ref{H3}).
Let $a,b\in\R_+$, $b-a\ge1$.
Then
\begin{enumerate}[(a)]
\item\label{l2-a}
for all $t,s\in[a,b]$ such that $\abs{t-s}\le1$,
\[
\left(\E(Y(t)-Y(s))^2\right)^{1/2}
\le K_3\abs{t-s}^{h_3}z(b),
\]
where $K_3=K_2^{1/2}\left(1+D^2\right)^{1/2}$.

\item\label{l2-b}
for all $t,s\in[a,b]$
\[
\left(\E(Y(t)-Y(s))^2\right)^{1/2}
\le K_3\abs{t-s}^{h_3}(b-a)^{h_5}z(b);
\]

\item\label{l2-c}
for all $t_1,t_2,s_1,s_2\in[a,b]$
\begin{multline*}
\left(\E\left(Y(t_1)-Y(t_2)-Y(s_1)+Y(s_2)\right)^2\right)^{1/2}\\
\le 2K_3 \max \left( \abs{t_1 -s_1}, \abs{t_2 -s_2 }\right)^{h_3}(b -a)^{h_5}z(b).
\end{multline*}
\end{enumerate}

\end{lemma}

\begin{proof}
It follows   from the  assumptions  (H\ref{H1}), (H\ref{H3}) and Lemma~\ref{l:1} that
\begin{equation}\label{bound2}
\E(Y(t)-Y(s))^2
\le K_2\abs{t-s}^{2H_t}+K_2D^2\abs{t-s}^{2\kappa}z^2(s).
\end{equation}

\eqref{l2-a}
In the case $\abs{t-s}\le1$, we have that
\begin{align*}
\E(Y(t)-Y(s))^2
&\le K_2\abs{t-s}^{2h_3}+K_2D^2\abs{t-s}^{2h_3}z^2(s)\\
&= K_2\abs{t-s}^{2h_3}\left(1+D^2z^2(s)\right)
\le K_3^2\abs{t-s}^{2h_3}z^2(b).
\end{align*}

\eqref{l2-b}
For arbitrary values of arguments, inequality \eqref{bound2} implies that
\begin{align*}
\E(Y(t)&-Y(s))^2\\
&\le K_2(b-a)^{2H_t}\abs{\frac{t-s}{b-a}}^{2H_t}+K_2D^2(b-a)^{2\kappa}\abs{\frac{t-s}{b-a}}^{2\kappa}z^2(s)\\
&\le K_2(b-a)^{2h_4}\abs{\frac{t-s}{b-a}}^{2h_3}(1+D^2z^2(s))\\
&\le K_3^2(b-a)^{2h_4-2h_3}\abs{t-s}^{2h_3}z^2(s).
\end{align*}

\eqref{l2-c}
Proof follows immediately from the part \eqref{l2-b}, since by Minkowski's inequality,
\begin{multline*}
\left(\E\left(Y(t_1)-Y(t_2)-Y(s_1)+Y(s_2)\right)^2\right)^{1/2}\\
\leq \left(\E (Y(t_1) -Y(s_1) )^2 \right)^{1/2} +
 \left(\E (Y(t_2) - Y(s_2) )^2 \right)^{1/2}.\qedhere
\end{multline*}

\end{proof}

\subsection{Asymptotic growth  of  the trajectories of mBm with probability 1}
Now we apply the results of Section \ref{sec:Gauss} to multifractional Brownian motion. The first result gives the maximal exponential bound for the weighted mBm. In order to get it, introduce the following notations. Let $b_k$, $k\ge0$, be  a sequence such that $b_0 =0$, $b_{k+1} -b_k \geq 1,$
and let
$a(t)>0$ be an increasing continuous function such that $a(t)\to\infty$ as $t\to\infty$,
 $a_k = a(b_k)$.

\begin{theorem}\label{th:6}
Let the Hurst function $H$ satisfy the conditions (H\ref{H1}) and (H\ref{H3}).
Assume that there exists $0<\gamma\le1$ such that
\begin{equation}\label{eq:cond-mBm}
\sum_{k=0}^\infty\frac{b_{k+1}^{h_2+\frac{\gamma h_4}{h_3}} \left(\log^2b_{k+1}+1\right)^{\frac{\gamma}{2h_3}}}{a_k}<\infty.
\end{equation}
Then for all  $0<\theta <1$, $u >0$
\begin{equation}\label{th6-1}
\P\set{\sup_{t>0}\frac{\abs{Y(t)}}{a(t)}>u}
\le2^{\frac{2}{h_3}-1}\exp\set{-\frac{u^2(1-\theta)^2}{2{A}^2}}A_7(\theta,\gamma),
\end{equation}
where
$A=K_1\sum_{k=0}^\infty\frac{b_{k+1}^{h_2}}{a_k}$,
\[
A_7(\theta,\gamma)=
\exp\set{\frac{K_1^{1-\frac{\gamma}{h_3}}K_3^{\frac{\gamma}{h_3}}}{\gamma {A}}
\sum_{k=0}^\infty\frac{b_{k+1}^{h_2+\frac{\gamma h_4}{h_3}} \left(\log^2b_{k+1}+1\right)^{\frac{\gamma}{2h_3}}}{a_k}
\left(\frac{2^{\frac{2}{h_3}-1}}{\theta^{\frac{1}{h_3}}}\right)^\gamma},
\]
or for all $u > {A}$
\begin{equation}\label{th6-2}
\P\set{\sup_{t>0}\frac{\abs{Y(t)}}{a(t)}>u}
\le2^{\frac{2}{h_3}-1}\sqrt{e}\exp\set{-\frac{u^2}{2{A}^2}}
A_7\left(1-\sqrt{1-\frac{{A}^2}{u^2}},\gamma\right).
\end{equation}
\end{theorem}

\begin{proof}
First, we check the conditions $(i)$--$(iii)$ of Theorem~\ref{th:2}.
Lemma~\ref{l:2}\eqref{l2-b} implies that the assumption $(i)$ holds with
$c_k=K_3(b_{k+1}-b_k)^{h_5}z(b_{k+1})$
and $\beta=h_3$.
According to \eqref{eq:variance}, we can choose
$m_k=K_1b_{k+1}^{h_2}$.
Then
\[
\sum_{k=0}^\infty\frac{m_k}{a_k}
=K_1\sum_{k=0}^\infty\frac{b_{k+1}^{h_2}}{a_k}<\infty,
\]
by~\eqref{eq:cond-mBm}, and the condition $(ii)$ is satisfied.
Finally, the condition $(iii)$ follows from~\eqref{eq:cond-mBm}, because in our case
\begin{align*}
 m_k^{1-\gamma/\beta}(b_{k+1}-b_k)^\gamma c_k^{\gamma/\beta}
&=K_1^{1-\frac{\gamma}{h_3}}K_3^{\frac{\gamma}{h_3}}
b_{k+1}^{h_2}(b_{k+1}-b_k)^{\gamma\left(1+\frac{h_5}{h_3}\right)}
\left(\log^2b_{k+1}+1\right)^{\frac{\gamma}{2h_3}}\\
&\le K_1^{1-\frac{\gamma}{h_3}}K_3^{\frac{\gamma}{h_3}}
b_{k+1}^{h_2+\frac{\gamma h_4}{h_3}}
\left(\log^2b_{k+1}+1\right)^{\frac{\gamma}{2h_3}},
\end{align*}
where we used the equality $h_5=h_4-h_3$.
Thus, the assumptions of Theorem~\ref{th:2} are satisfied.
Now the statements \eqref{th6-1} and \eqref{th6-2} follow from Corollaries \ref{cor:3} and \ref{cor:4} respectively.
\end{proof}
Now we present the first main result of this section, namely,  the power upper bound for the asymptotic growth of the trajectories of mBm with probability 1.

 \begin{theorem}\label{l:bound-mBm}
For any $\delta>0$ there exists a nonnegative random variable $\xi=\xi(\delta)$ such that for all $t>0$
\begin{equation}\label{eq:bound-proc}
\abs{Y(t)}\le \left(t^{h_2+\delta}\vee1\right)\xi
\qquad\text{a.\,s.},
\end{equation}
and there exist positive constants $C_1=C_1(\delta)$ and $C_2=C_2(\delta)$ such that for all $u>0$
\[
\P(\xi>u)\le C_1e^{-C_2u^2}.
\]
\end{theorem}
\begin{proof}
Put in Theorem~\ref{th:6}
$a(t)=t^{h_2+\delta}\vee1$, $b_0=0$, $b_k=e^k$, $k\ge1$,
and arbitrary
$\gamma\in\left(0,\frac{\delta h_3}{h_4}\wedge1\right)$,
$\theta\in(0,1)$.
Then $a_0=1$, $a_k=e^{k(h_2+\delta)}$, $k\ge1$,
\begin{multline*}
\sum_{k=0}^\infty\frac{b_{k+1}^{h_2+\frac{\gamma h_4}{h_3}} \left(\log^2b_{k+1}+1\right)^{\frac{\gamma}{2h_3}}}{a_k}\\
=e^{h_2+\frac{\gamma h_4}{h_3}}
\left(2^{\frac{\gamma}{2h_3}}+\sum_{k=1}^\infty e^{k\left(\frac{\gamma h_4}{h_3}-\delta\right)}\left((k+1)^2+1\right)^{\frac{\gamma}{2h_3}}\right)<\infty.
\end{multline*}
Now the result follows from Theorem~\ref{th:6}, if we additionally put
\[
\xi=\sup_{t>0}\frac{\abs{Y(t)}}{t^{h_2+\delta}\vee1},
\quad
C_1=2^{\frac{2}{h_3}-1}A_7(\theta,\gamma),
\quad
C_2=\frac{(1-\theta)^2}{2{A}^2}.\qedhere
\]
\end{proof}

\subsection{Asymptotic growth with probability 1 of  the increments of mBm}
Let  $\Delta\in(0,1]$.
Consider the increment of mBm $Z(\ve t)=Y(t_1)-Y(t_2)$, $\ve t\in\T_\Delta$.
Let $b_k$, $k\ge0$, be  a sequence such that $b_0 =0$, $b_{k+1} -b_k \geq 1,$
and let
$a(t)>0$ be an increasing continuous function such that $a(t)\to\infty$ as $t\to\infty$,
 $a_k = a(b_k)$.

\begin{theorem}\label{th:5}
Let the Hurst function $H$ satisfy the conditions (H\ref{H1}) and (H\ref{H3}).
Assume that there exists $0<\gamma\le1$ such that
\begin{equation}\label{eq:1series}
\sum_{l=0}^\infty\frac{b_{l+1}^{\frac{2h_5\gamma}{h_3}}z(b_{l+1})}{a_l}<\infty.
\end{equation}
Then for all  $\theta \in(0,1)$, $\eps\in(0,h_3)$ and $\lambda>0$
\[
\E \exp\set{ \lambda \; \sup_{ \ve{t} \in \T_\Delta}
\frac {\abs{Z(\ve{t})}}{a(t_1)}}
 \le\frac1\Delta\exp \set{ \frac{\lambda^2A^2\Delta^{2h_3}}{2(1- \theta)^2}}
A_8(\theta,\gamma,\eps),
\]
where
$A = K_3\sum_{l=0}^\infty \frac{z(b_{l+1})}{a_l}$,
\begin{multline*}
A_8(\theta,\gamma,\eps)=2^{\frac2\eps+2}\exp\set{\frac{K_3}{A}
\sum_{l=0}^\infty\frac{z(b_{l+1})\log(b_{l+1}-b_l)}{a_l}}\\
\times\exp\set{\frac{K_3}{\gamma A4^{2\gamma}\left(1-\frac{\eps}{h_3}\right)^\frac{2\gamma}{\eps}\theta^{\frac{2\gamma}{h_3}}}
\sum_{l=0}^\infty\frac{z(b_{l+1})(b_{l+1}-b_l)^{\frac{2h_5\gamma}{h_3}}}{a_l}}.
\end{multline*}
\end{theorem}

\begin{proof}
We need to verify the assumptions of Theorem~\ref{th:4} for the process $Z$.
By Lemma~\ref{l:2}\eqref{l2-a}, for all $\ve t\in\T_{b_l,b_{l+1},\Delta}$
\begin{align*}
\left(\E(Z(\ve t))^2\right)^{\frac12}
&=\left(\E(Y(t_1)-Y(t_2))^2\right)^{\frac12}
\le K_3z(b_{l+1})\Delta^{h_3}.
\end{align*}
Hence, the condition $(vi)$ is satisfied with
$m_l=K_3z(b_{l+1})\Delta^{h_3}$.
Further, Lemma~\ref{l:2}\eqref{l2-c} implies
\[\sup_{\substack{ d(\ve{t}, \ve{s}) \leq h,\\ \ve{t}, \ve{s} \in {\T}_{b_l, b_{l+1},\Delta}}}
\left(\E ( Z (\ve{t}) - Z(\ve{s}))^2 \right)^{\frac{1}{2}} \leq 2K_3(b_{l+1}-b_l)^{h_5}z(b_{l+1}) h^{h_3}.
\]
Thus, the condition $(vii)$ holds true with
$c_l=2K_3(b_{l+1}-b_l)^{h_5}z(b_{l+1})$,
$\beta=h_3$.
It is not hard to see that in this case the condition $(viii)$ is equivalent to the condition
\begin{gather*}
\sum_{l=0}^\infty \frac{z(b_{l+1})}{a_l} < \infty, \quad
\sum_{l=0}^\infty\frac{z(b_{l+1})\log(b_{l+1}-b_l)}{a_l} <\infty,\\
\sum_{l=0}^\infty\frac{z(b_{l+1}) (b_{l+1}-b_l)^{\frac{2h_5\gamma}{h_3}}}{a_l}<\infty.
\end{gather*}
Obviously, these three series converge when \eqref{eq:1series} holds.
Now the result follows from Theorem~\ref{th:4}.
\end{proof}

Let $d_k$, $k\ge0$, be  a strictly decreasing sequence such that
$d_0 =1$, $d_k\downarrow0$ as\linebreak $k\to\infty$.
Let $g:(0,1]\to(0,\infty)$ be a continuous function and $g_k$, $k\ge0$, be such a sequence that
$0<g_k\le\min_{d_{k+1}\le t\le d_k}g(t)$.

\begin{theorem}\label{th:7}
Assume that the assumptions of Theorem~\ref{th:5} hold and
\[
\sum_{k=0}^{\infty}\frac{d_k^{h_3}\abs{\log d_k}}{g_k}<\infty.
\]
Then for all  $\theta \in(0,1)$, $\eps\in(0,h_3)$ and $\lambda>0$
\[
I(\lambda)=
\E \exp\set{ \lambda \; \sup_{0\le t_2<t_1\le t_2+1}
\frac {\abs{Z(\ve{t})}}{a(t_1)g(t_1-t_2)}}
 \le\exp\set{ \frac{\lambda^2A^2B^2}{2(1- \theta)^2}}
 A_9(\theta,\gamma,\eps),
 \]
where
\[
B=\sum_{k=0}^{\infty}\frac{d_k^{h_3}}{g_k},\qquad
A_9(\theta,\gamma,\eps)=\exp\set{\frac1B\sum_{k=0}^{\infty}\frac{d_k^{h_3}\abs{\log d_k}}{g_k}}
A_8(\theta,\gamma,\eps).
\]
\end{theorem}
\begin{proof}
Denote
$\T^{(k)}=\set{(t_1,t_2)\in\R_+ : d_{k+1}<t_1-t_2\le d_k}$.
Then $\T^{(k)}\subset\T_{d_k}$
and $\set{(t_1,t_2)\in\R : 0\le t_2<t_1\le t_2+1}=\bigcup_{k=0}^\infty\T^{(k)}$.
Therefore
\begin{align*}
\sup_{0\le t_2<t_1\le t_2+1}
\frac {\abs{Z(\ve{t})}}{a(t_1)g(t_1-t_2)}
&\le\sum_{k=0}^\infty\sup_{\ve t\in\T^{(k)}}
\frac {\abs{Z(\ve{t})}}{a(t_1)g(t_1-t_2)}\\
&\le\sum_{k=0}^\infty\frac{1}{g_k}\;\sup_{\ve t\in\T^{(k)}}
\frac {\abs{Z(\ve{t})}}{a(t_1)}
\le\sum_{k=0}^\infty\frac{1}{g_k}\;\sup_{\ve t\in\T_{d_k}}
\frac {\abs{Z(\ve{t})}}{a(t_1)}.
\end{align*}
Let $r_k>0$,
$\sum_{k=0}^\infty\frac{1}{r_k}=1$.
Then
\[
I(\lambda)\le\E\exp\set{\sum_{k=0}^\infty\frac{\lambda}{g_k}\,\sup_{\ve t\in\T_{d_k}}
\frac {\abs{Z(\ve{t})}}{a(t_1)}}
\le\prod_{k=0}^\infty\left(\E\exp\set{\frac{\lambda r_k}{g_k}\,\sup_{\ve t\in\T_{d_k}}
\frac {\abs{Z(\ve{t})}}{a(t_1)}}\right)^{\frac{1}{r_k}}.
\]
By Theorem~\ref{th:5}, we get
\begin{align*}
I(\lambda)&\le\prod_{k=0}^\infty\left(\frac1{d_k}
\exp\set{\frac{\lambda^2r_k^2A^2d_k^{2h_3}}{2(1- \theta)^2g_k^2}}
A_8(\theta,\gamma,\eps)\right)^{\frac{1}{r_k}}\\
&=A_8(\theta,\gamma,\eps)
\exp\set{\frac{\lambda^2A^2}{2(1- \theta)^2}\sum_{k=0}^\infty\frac{r_kd_k^{2h_3}}{g_k^2}}
\exp\set{-\sum_{k=0}^\infty\frac{\log d_k}{r_k}}.
\end{align*}
Put $r_k=\frac{Bg_k}{d_k^{h_3}}$.
Then
\[
I(\lambda)\le
A_8(\theta,\gamma,\eps)
\exp\set{\frac{\lambda^2A^2B^2}{2(1- \theta)^2}}
\exp\set{-\frac1B\sum_{k=0}^\infty\frac{d_k^{h_3}\log d_k}{g_k}}.\qedhere
\]
\end{proof}

\begin{corollary}\label{cor:d1}
Let the assumptions of Theorem~\ref{th:7} hold.
Then for  all   $\theta\in(0,1)$, $\eps\in(0,h_3)$ and  $u>0$,
\begin{equation}\label{eq:b1}
\P\set{\sup_{0\le t_2<t_1\le t_2+1}
\frac {\abs{Z(\ve{t})}}{a(t_1)g(t_1-t_2)}>u}\le
\exp\set{-\frac{u^2(1-\theta)^2}{2A^2B^2}}
A_9(\theta,\gamma,\eps).
\end{equation}

Indeed, by Chebyshev's inequality,
\begin{align*}
\P&\set{\sup_{0\le t_2<t_1\le t_2+1}
\frac {\abs{Z(\ve{t})}}{a(t_1)g(t_1-t_2)}>u}\\
&\qquad\le \frac1{e^{\lambda u}}\E \exp\set{ \lambda \; \sup_{0\le t_2<t_1\le t_2+1}
\frac {\abs{Z(\ve{t})}}{a(t_1)g(t_1-t_2)}}\\
&\qquad\le\exp\set{ \frac{\lambda^2A^2B^2}{2(1- \theta)^2}-\lambda u}
 A_9(\theta,\gamma,\eps).
\end{align*}
If we put $\lambda=\frac{u(1-\theta)^2}{A^2B^2}$, we get \eqref{eq:b1}.
\end{corollary}

With the help of Corollary~\ref{cor:d1}, we can now state the second main result of this section, which
is the following upper bound for the asymptotic growth of the increments of mBm with probability 1.

\begin{theorem}\label{l:bound-incr}
For any $\eps>0$ and any $p>2$ there exists a nonnegative random variable $\eta=\eta(\eps,p)$ such that for all $0\le t_2<t_1\le t_2+1$
\begin{equation}\label{bound-incr}
\abs{Z(\ve{t})}\le \left(t_1^{h_2+\eps}\vee1\right)(t_1-t_2)^{h_3}\left(\abs{\log(t_1-t_2)}^p\vee1\right)\eta
\qquad\text{a.\,s.},
\end{equation}
and there exist positive constants $C_1=C_1(\eps,p)$ and $C_2=C_2(\eps,p)$ such that for all $u>0$
\[
\P(\eta>u)\le C_1e^{-C_2u^2}.
\]
\end{theorem}

\begin{proof}
Put in Theorem~\ref{th:5}
$a(t)=t^{h_2+\eps}\vee1$, $b_0=0$, $b_l=e^l$, $l\ge1$.
Then $a_0=1$, $a_l=e^{l(h_2+\eps)}$, $l\ge1$, and
\[
\sum_{l=0}^\infty\frac{b_{l+1}^{\frac{2h_5\gamma}{h_3}}z(b_{l+1})}{a_l}
=e^{\frac{2h_5\gamma}{h_3}+h_2}\left(\sqrt2+\sum_{l=1}^\infty e^{l\left(\frac{2h_5\gamma}{h_3}-\eps\right)}\left((l+1)^2+1\right)^{1/2}\right).
\]
Therefore, \eqref{eq:1series} holds, if we choose $\gamma\in\left(0,\frac{\eps h_3}{2h_5}\wedge1\right)$.

Further, put in Theorem~\ref{th:7}
$g(t)=t^{h_3}\abs{\log t}^p$,
$d_k=e^{-k}$, $k\ge0$,
$g_0=e^{-h_3}$,
$g_k=d_{k+1}^{h_3}\abs{\log d_k}^p=e^{-(k+1)h_3}k^p$, $k\ge1$.
We have
\begin{gather*}
\sum_{k=0}^{\infty}\frac{d_k^{h_3}}{g_k}
=e^{h_3}\left(1+\sum_{k=1}^{\infty}\frac{1}{k^p}\right)<\infty,
\\
\sum_{k=0}^{\infty}\frac{d_k^{h_3}\abs{\log d_k}}{g_k}
=e^{h_3}\sum_{k=1}^{\infty}\frac{1}{k^{p-1}}<\infty.
\end{gather*}
Thus, the conditions of Theorem~\ref{th:7} are satisfied.
The result follows from Corollary~\ref{cor:d1}, if we  put
\[
\eta=\sup_{0\le t_2<t_1\le t_2+1}
\frac {\abs{Z(\ve{t})}}{\left(t_1^{h_2+\eps}\vee1\right)(t_1-t_2)^{h_3}\left(\abs{\log(t_1-t_2)}^p\vee1\right)}.\qedhere
\]
\end{proof}

\subsection{Pathwise integration with respect to multifractional Brownian motion}\label{integration}
To describe the statistical model, we need to introduce the pathwise integrals w.r.t. mBm.
 Consider
two non-random functions
 $f$ and $g$  defined on some interval $[a,b]\subset \R^+$.
Let $\alpha>0$.
Denote the Riemann--Liouville left- and right-sided
fractional integrals on $(a, b)$ of order $\alpha$ by
\[
(I_{a+}^{\alpha}f)(s):=\frac{1}{\Gamma(\alpha)}\int_{a}^{s}f(t)(s-t)^{\alpha-1}dt,
\]
and
\[
(I_{b-}^{\alpha}g)(s):=\frac{1}{\Gamma(\alpha)}\int_{s}^{b}g(t)(t-s)^{\alpha-1}dt,
\]
respectively.
 Suppose also
that the  the following limits exist:
\[
f(u+):=\lim_{\delta\downarrow0}f(u+\delta)\text{
and } g(u-):=\lim_{\delta\downarrow0}g(u-\delta),\   a\leq u\leq
b.
\]
Let $f_{a+}(s):=(f(s)-f(a+))1_{(a,b)}(s), \
g_{b-}(s):=(g(b-)-g(s))1_{(a,b)}(s).$
Suppose that
$f_{a+}\in I_{a+}^{\alpha}(L_p[a,b])), \ g_{b-}\in
I_{b-}^{1-\alpha}(L_q[a,b]))$ for some $p\geq 1, \ q\geq 1,
1/p+1/q\leq 1, \ 0\leq \alpha \leq 1.$
Introduce the fractional derivatives
\begin{equation*}\label{eq1.1}(\mathcal{D}_{a+}^{\alpha}f_{a+})(s)=\frac{1}{\Gamma(1-\alpha)}\Big(\frac{f_{a+}(s)}{(s-a)^\alpha}+\alpha
\int_{a}^s\frac{f_{a+}(s)-f_{a+}(u)}{(s-u)^{1+\alpha}}du\Big)1_{(a,b)}(s)\end{equation*}
\begin{equation*}\label{eq1.2}(\mathcal{D}_{b-}^{1-\alpha}g_{b-})(s)=\frac{e^{-\emph{i}\pi
\alpha}}{\Gamma(\alpha)}\Big(\frac{g_{b-}(s)}{(b-s)^{1-\alpha}}+(1-\alpha)
\int_{s}^b\frac{g_{b-}(s)-g_{b-}(u)}{(s-u)^{2-\alpha}}du\Big)1_{(a,b)}(s).\end{equation*}
It is known that
 $\mathcal{D}_{a+}^{\alpha}f_{a+}\in L_p[a,b], \ \mathcal{D}_{b-}^{1-\alpha}g_{b-}\in
L_q[a,b].$
Under above
assumptions, the generalized (fractional) Lebesgue-Stieltjes
integral $\int_a^bf(x)dg(x)$ is defined as
$$\int_a^bf(x)dg(x):=e^{\emph{i}\pi
\alpha}\int_a^b(\mathcal{D}_{a+}^{\alpha}f_{a+})(x)(\mathcal{D}_{b-}^{1-\alpha}g_{b-})(x)dx+f(a+)(g(b-)-g(a+)),$$
and for $\alpha p<1$ it can be simplified to
$$\int_a^bf(x)dg(x):=e^{\emph{i}\pi
\alpha}\int_a^b(\mathcal{D}_{a+}^{\alpha}f)(x)(\mathcal{D}_{b-}^{1-\alpha}g_{b-})(x)dx,$$
see \cite{zele,Zah99}.

Assume that the Hurst function satisfies the conditions (H\ref{H1})--(H\ref{H3}) and, additionally, $h_3=\min\set{h_1,\kappa}>1/2$.  In this case, according to
Remark~\ref{remark1}, process  $Y$ with probability 1 has H\"{o}lder trajectories up to order $h_3$  on any finite interval $[0,T]$.
 As  follows from \cite{Samko}, for any $1-h_3<\alpha<1$ there
exists  the fractional derivative $\mathcal{D}_{b-}^{1-\alpha}Y_{b-}\in
L_{\infty}[a,b]$ for any $0\leq a<b\le T.$
Let we have another process, say $Z=\{Z_t, t\in [0,T]\}$, also having H\"{o}lder trajectories up to some order $h$ with $h+h_3>1$. In particular, it can be $h=h_3.$ Then, according to \cite{zele}, there exists  an  integral $\int_a^b Z_sdY_s$, which is the limit a.s. of the Riemann sums and   has the standard properties (so called path-wise integral).
This integral is defined as
\begin{equation}\label{eq2.1} \int_a^bZdY:=e^{\emph{i}\pi
\alpha}\int_a^b(\mathcal{D}_{a+}^{\alpha}Z)(x)(\mathcal{D}_{b-}^{1-\alpha}Y_{b-})(x)dx.\end{equation}
An  evident estimate follows immediately from \eqref{eq2.1}:
\begin{equation}\label{eq2.2}\Big|\int_a^bZ\,dY\Big|\leq\sup_{a\leq x\leq b}|(\mathcal{D}_{b-}^{1-\alpha}Y_{b-})(x)|\int_a^b|(\mathcal{D}_{a+}^{\alpha}Z)(x)|dx.\end{equation}

\section{Drift parameter estimation in stochastic differential equations driven by mBm}\label{sec:estimation}

\subsection{Linear model}

Consider the process
\begin{equation}\label{eq:linear}
X_t=\theta t+Y_t,\quad t\ge0,
\end{equation}
where $\theta\in\R$ is an unknown parameter, $Y_t$ is an mBm with the Hurst function $H_t$ satisfying  the conditions (H\ref{H1})--(H\ref{H3}).
Assume that our aim is to estimate the parameter $\theta$ by the observations of $X_t$.
Let us introduce the estimator
\[
\hat\theta_T=\frac{X_T}{T}=\theta+\frac{Y_T}{T}.
\]

\begin{theorem}
\begin{enumerate}[1)]
\item The estimator $\hat\theta_T$ is strongly consistent as $T\to\infty$.
\item For all $T>0$,
\[
\frac{T^{1-H_T}}{C(H_T)}\left(\hat\theta_T-\theta\right)
\simeq N(0,1),
\]
where
$C(H)=\left(\frac{\pi}{H\Gamma(2H)\sin(\pi H)}\right)^{1/2}$.
Consequently, a confidence interval of level $1-\alpha$ is given by
\[
\hat\theta_T\pm\frac{C(H_T)}{T^{1-H_T}}z_{1-\alpha/2},
\]
where $z_p$ denotes the $p$-quantile of the standard normal distribution.
\end{enumerate}
\end{theorem}

\begin{proof}
1) By Theorem~\ref{l:bound-mBm}, for all $T>1$ and $\delta>0$
\[
\abs{Y_T}\le T^{h_2+\delta}\xi \quad\text{a.\,s.},
\]
where $\xi=\xi(\delta)$ is some nonnegative random variable.
Hence, if we choose $\delta<1-h_2$, then we get
\[
\frac{\abs{Y_T}}{T}\le\frac{\xi}{T^{1-h_2-\delta}}\to0, \quad\text{a.\,s.\ as }T\to\infty.
\]

2) Note that one-dimensional distributions of mBm $Y_t$ are centered Gaussian with standard deviation
$C(H_t)t^{H_t}$, see~\eqref{eq:var-mBm}.
Therefore,
\[
\frac{T^{1-H_T}}{C(H_T)}\left(\hat\theta_T-\theta\right)
=\frac{Y_T}{C(H_T)T^{H_T}}
\simeq N(0,1).\qedhere
\]
\end{proof}

\subsection{Multifractional Ornstein--Uhlenbeck process}

Let, as in subsection~\ref{integration}, $h_3>1/2$. In this subsection we consider the estimation of the unknown parameter $\theta\ge0$ by observations of the process $X=\{X_t,t\geq 0\}$ that is a solution of the stochastic differential equation of Langevin type,
\begin{equation}\label{eq:equation}
X_t=x_0+\theta\int_0^tX_s\,ds+Y_t,
\end{equation}
where $x_0\in\R$ is a known constant, $Y=\{Y_t,t\geq 0\}$ is an mBm.
This solution exists and is unique, see~\cite[Th.~4.1]{Approx-mBm}.

Note that the trajectories of the processes $Y$ and consequently $X$ are a.\,s.\ H\"older continuous up to order $h_3$.
Therefore, according to subsection~\ref{integration}, path-wise integrals $\int_0^TX_s\,dX_s$ and $\int_0^TX_s\,dY_s$ are well defined.
One can verify that the solution of~\eqref{eq:equation} can be  represented in the following form
\[
X_t=e^{\theta t}\left(x_0+\int_0^te^{-\theta s}\,dY_s\right).
\]
Using the integration-by-parts, this process can be written as follows
\begin{equation}\label{eq:solution}
X_t=x_0e^{\theta t}+\theta e^{\theta t}\int_0^te^{-\theta s}Y_s\,ds+Y_t.
\end{equation}
We call the process $X=\{X_t,t\geq 0\}$ \emph{multifractional Ornstein--Uhlenbeck process.}

Let, more precisely, our goal be to estimate the unknown drift parameter $\theta\in \R$ by the continuous-time observations on the interval $[0,T]$.
Consider the estimator
\begin{equation}\label{eq:estimator}
\hat\theta_T=\frac{\int_0^TX_s\,dX_s}{\int_0^TX_s^2\,ds}.
\end{equation}

\begin{rem}
In the case of the equation driven by ordinary fBm the estimator~\eqref{eq:estimator} was studied in~\cite{BESO,hunu,KMM}.
Hu and Nualart \cite{hunu} proved that in the ergodic case ($\theta<0$) it is strongly consistent for all $H\ge\frac12$ and asymptotically normal for $H\in[\frac12,\frac34)$.
They considered $\int_0^TX_t\,dX_t$ in~\eqref{eq:estimator} as a divergence-type integral.
In \cite{BESO,KMM} the corresponding non-ergodic case $\theta>0$ was investigated and the strong consistency of the estimator \eqref{eq:estimator} was proved for $H\ge\frac12$.
It was   also obtained in \cite{BESO}  that
$e^{\theta t}\left(\widehat\theta_t-\theta\right)$
converges in law to $2\theta\mathcal C(1)$ as $t\to\infty$, where $\mathcal C(1)$ is the standard Cauchy distribution.
In~\cite{EMESO} the more general situation was studied, namely the non-ergodic Ornstein-Uhlenbeck process driven by a Gaussian process.
\end{rem}

Since by \eqref{eq:equation},
$dX_s=\theta X_s\,ds+dY_s$,
we have that $\hat\theta_T$ admits the following stochastic representation
\[
\hat\theta_T=\theta+\frac{\int_0^TX_s\,dY_s}{\int_0^TX_s^2\,ds}.
\]

Denote by $\mathfrak Z$ a class of random variables $\zeta\ge0$  with the following property: there exist positive constants $C_1$ and $C_2$ not depending on $T$ such that for all $u>0$
\[
\P(\zeta>u)\le C_1e^{-C_2u^2}.
\]

\begin{lemma}\label{l:numerator}
Let $\eps>0$, $T>1$, $\theta>0$.
Then there exists such $\zeta\in\mathfrak Z$ that
\begin{equation}\label{eq:numer-1}
\abs{\int_0^TX_s\,dY_s}\le\zeta^2
T^{h_2+\eps+1}e^{\theta T}.
\end{equation}
\end{lemma}

\begin{proof}
By~\eqref{eq:solution},
\begin{equation}\label{eq:c1}
\sup_{0\le s\le t}\abs{X_s}\le\abs{x_0}e^{\theta t}+\theta e^{\theta t}\int_0^te^{-\theta s}\sup_{0\le u\le s}\abs{Y_u}\,ds+\sup_{0\le s\le t}\abs{Y_s}.
\end{equation}
Then \eqref{eq:equation} implies that for $t_1>t_2\ge0$
\begin{equation}\label{eq:c2}
\begin{split}
\abs{X_{t_1}-X_{t_2}}
&\le\theta\int_{t_2}^{t_1}\left(\abs{x_0}e^{\theta s}+\theta e^{\theta s}\int_0^se^{-\theta v}\sup_{0\le u\le v}\abs{Y_u}\,dv+\sup_{0\le u\le s}\abs{Y_u}\right)ds\\
&\quad+\abs{Y_{t_1}-Y_{t_2}}.
\end{split}
\end{equation}
Furthermore, using Theorems~\ref{l:bound-mBm} and \ref{l:bound-incr} we get for $t\ge0$ and $\delta>0$
\begin{equation}\label{eq:c3}
\sup_{0\le s\le t}\abs{Y_s}\le\left(t^{h_2+\delta}+1\right)\xi
\qquad\text{a.\,s.},
\end{equation}
and for $0\le t_2<t_1\le t_2+1$
\begin{equation}\label{eq:c4}
\begin{split}
\abs{Y_{t_1}-Y_{t_2}}&\le
\left(t_1^{h_2+\eps}+1\right)(t_1-t_2)^{h_3}\left(\abs{\log(t_1-t_2)}^p+1\right)\eta\\
&=\left(t_1^{h_2+\eps}+1\right)\left((t_1-t_2)^{h_3}\abs{\log(t_1-t_2)}^p+(t_1-t_2)^{h_3}\right)\eta\\
&\le C\left(t_1^{h_2+\eps}+1\right)(t_1-t_2)^{h_3-r}\eta
\qquad\text{a.\,s.},
\end{split}
\end{equation}
where $0<r<h_3-1/2$.
Then by \eqref{eq:c3},
\[
\int_0^te^{-\theta s}\sup_{0\le u\le s}\abs{Y_u}\,ds
\le\xi\int_0^te^{-\theta s}\left(s^{h_2+\delta}+1\right)\,ds
\le C\xi
\]
Therefore, from \eqref{eq:c1} we obtain
\[
\sup_{0\le s\le t}\abs{X_s}\le\abs{x_0}e^{\theta t}+\theta e^{\theta t}C\xi+\left(t^{h_2+\delta}+1\right)\xi.
\]
It follows from \eqref{eq:c2} and \eqref{eq:c4} that
\begin{align*}
\abs{X_{t_1}-X_{t_2}}
&\le\theta\int_{t_2}^{t_1}\left(\abs{x_0}e^{\theta s}+\theta e^{\theta s}C\xi+\left(s^{h_2+\delta}+1\right)\xi\right)ds\\
&\quad+ C\left(t_1^{h_2+\eps}+1\right)(t_1-t_2)^{h_3-r}\eta.
\end{align*}
These formulas can be rewritten using simplified notation as follows:
 \begin{align}
 \sup_{0\le s\le t}\abs{X_s}&\le\left(e^{\theta t}+t^{h_2+\delta}\right)\zeta,\label{eq:c5}\\
\abs{X_{t_1}-X_{t_2}}
&\le\zeta\left(e^{\theta t_1}+t_1^{h_2+\delta}\right)(t_1-t_2)+ \zeta\left(t_1^{h_2+\eps}+1\right)(t_1-t_2)^{h_3-r}.\label{eq:c6}
\end{align}
In order to estimate $\int_0^TX_s\,dY_s$ we write
\begin{multline}\label{eq:c7}
\abs{\int_0^TX_s\,dY_s}\le\sum_{k=0}^{[T]+1}
\abs{\int_k^{k+1}X_s\,dY_s}\\
\le\sum_{k=0}^{[T]+1}\sup_{k\le s\le k+1}
\abs{\left(D^{1-\alpha}_{k+1-}Y_{k+1-}\right)(s)}
\abs{\int_k^{k+1}\left(D^{\alpha}_{k+}X_{k+}\right)(s)\,ds},
\end{multline}
where  $1-h_3+r<\alpha<h_3-r$, see~\eqref{eq2.2}.
Now we need to estimate the fractional derivatives. By~\eqref{eq:c4},
\begin{equation}\label{eq:c8}
\begin{split}
&\abs{\left(D^{1-\alpha}_{k+1-}Y_{k+1-}\right)(s)}\\
&\quad\le\frac1{\Gamma(\alpha)}\left(\frac{\abs{Y_{k+1}-Y_s}}{(k+1-s)^{1-\alpha}}
+(1-\alpha)\int_s^{k+1}\frac{\abs{Y_u-Y_s}}{(u-s)^{2-\alpha}}\,du\right)\\
&\quad\le\zeta\Biggl(\left((k+1)^{h_2+\eps}+1\right)(k+1-s)^{h_3-r-1+\alpha}\\
&\quad\quad+\int_s^{k+1}\left(u^{h_2u+\eps}+1\right)(u-s)^{h_3-r-2+\alpha}\,du\Biggr)\\
&\quad\le\zeta\left((k+1)^{h_2+\eps}+1\right)
\left((k+1-s)^{h_3-r-1+\alpha}
+\int_s^{k+1}(u-s)^{h_3-r-2+\alpha}\,du\right)\\
&\quad\le\zeta\left((k+1)^{h_2+\eps}+1\right)(k+1-s)^{h_3-r-1+\alpha}\\
&\quad\le\zeta\left((k+1)^{h_2+\eps}+1\right).
\end{split}
\end{equation}
Applying \eqref{eq:c5}--\eqref{eq:c6}, we get
\begin{align*}
&\abs{\left(D^{\alpha}_{k+}X\right)(s)}
\le\frac1{\Gamma(1-\alpha)}\left(\frac{\abs{X_s}}{(s-k)^{\alpha}}
+\alpha\int_k^s\frac{\abs{X_s-X_u}}{(s-u)^{\alpha+1}}\,du\right)\\
&\quad\le\zeta\Biggl(\left(e^{\theta s}+s^{h_2+\delta}\right)(s-k)^{-\alpha}\\
&\quad\quad+\int_k^s\left(\left(e^{\theta s}+s^{h_2+\delta}\right)(s-u)^{-\alpha}
+\left(s^{h_2+\eps}+1\right)(s-u)^{h_3-r-\alpha-1}\right)\,du\Biggr)\\
&\quad\le\zeta\left(\left(e^{\theta s}+s^{h_2+\delta}\right)\left((s-k)^{-\alpha}+(s-k)^{1-\alpha}\right)
+\left(s^{h_2+\eps}+1\right)(s-k)^{h_3-r-\alpha}\right).
\end{align*}
Then
\begin{equation}\label{eq:c9}
\int_k^{k+1}\abs{\left(D^{\alpha}_{k+}X\right)(s)}\,ds
\le\zeta\left(e^{\theta (k+1)}+(k+1)^{h_2+\delta}
+(k+1)^{h_2+\eps}+1\right).
\end{equation}
Combining \eqref{eq:c7}--\eqref{eq:c9}, we get \begin{multline*}
\abs{\int_0^TX_s\,dY_s}\le\zeta^2\sum_{k=0}^{[T]+1}
\left((k+1)^{h_2+\eps}+1\right)\\
\times\left(e^{\theta (k+1)}+(k+1)^{h_2+\delta}
+(k+1)^{h_2+\eps}+1\right).
\end{multline*}
Now, each summand in the right-hand side of the latter inequality can be bounded by $C T^{h_2+\eps}e^{\theta T}$, whence~\eqref{eq:numer-1} follows.
\end{proof}

\begin{theorem}
Let $\theta>0$.
Then the estimator $\hat\theta_T$ is strongly consistent as $T\to\infty$.
\end{theorem}

\begin{proof}
By \eqref{eq:numer-1},
\[
\abs{\hat\theta_T-\theta}=\frac{\abs{\int_0^TX_s\,dY_s}}{\int_0^TX_s^2\,ds}
\le\zeta^2\frac{T^{h_2+\eps+1}e^{\theta T}}{\int_0^TX_s^2\,ds}.
\]
Applying L'H\^opital's rule and~\eqref{eq:solution}, we get
\begin{equation}\label{eq:lim1}
\begin{split}
\lim_{T\to\infty}\frac{T^{h_2+\eps+1}e^{\theta T}}{\int_0^TX_s^2\,ds}
&=\lim_{T\to\infty}\frac{\left((h_2+\eps+1)T^{h_2+\eps}+T^{h_2+\eps+1}\theta\right)e^{\theta T}}{X_T^2}\\
&=\lim_{T\to\infty}\frac{\left((h_2+\eps+1)T^{h_2+\eps}+T^{h_2+\eps+1}\theta\right)}{e^{\theta T}\left(x_0+\theta\int_0^Te^{-\theta s}Y_s\,ds+e^{-\theta T}Y_T\right)^2}.
\end{split}
\end{equation}
Using the bound~\eqref{eq:bound-proc}, we obtain that
$e^{-\theta T}Y_T\to0$ a.\,s.\ as $T\to\infty$.
Moreover, with probability 1 there exists the limit
$\lim_{T\to\infty}\int_0^Te^{-\theta s}Y_s\,ds
=\int_0^\infty e^{-\theta s}Y_s\,ds$.
Obviously, this limit is a Gaussian random variable.
This implies that
\[
\lim_{T\to\infty}\left(x_0+\theta\int_0^Te^{-\theta s}Y_s\,ds+e^{-\theta T}Y_T\right)^2
=\left(x_0+\theta\int_0^\infty e^{-\theta s}Y_s\,ds\right)^2>0\quad\text{a.\,s.}
\]
Therefore, it follows from~\eqref{eq:lim1} that
\[
\lim_{T\to\infty}\frac{T^{h_2+\eps+1}e^{\theta T}}{\int_0^TX_s^2\,ds}=0\quad\text{a.\,s.}
\]
This completes the proof.
\end{proof}

\begin{theorem}
Let $\theta=0$.
Then the estimator $\hat\theta_T$ is consistent as $T\to\infty$.
\end{theorem}

\begin{proof}
In this case $X_t=x_0+Y_t$.
Hence,
\[
\int_0^TX_s\,dY_s=x_0Y_T+\int_0^TY_s\,dY_s
=x_0Y_T+\frac12Y_T^2,
\]
see formula~(32) in~\cite{zele}.
Then for $T>1$,
\[
\abs{\hat\theta_T-\theta}=\frac{\abs{\int_0^TX_s\,dY_s}}{\int_0^TX_s^2\,ds}
\le\frac{x_0\abs{Y_T}+\frac12Y_T^2}{\int_0^TX_s^2\,ds}
\le\zeta^2\frac{T^{2h_2+2\delta}}{\int_0^TX_s^2\,ds},
\]
by Theorem~\ref{l:bound-mBm}.
It follows from the Cauchy--Schwarz inequality that
\begin{align*}
\int_0^TX_s^2\,ds
&\ge\frac{1}{T}\left(\int_0^TX_s\,ds\right)^2
=\frac{1}{T}\left(\int_0^T(x_0+Y_s)\,ds\right)^2
=T\left(x_0+\frac1T\int_0^TY_s\,ds\right)^2\\
&=T\left(x_0+\sigma_T \mathcal N(0,1)\right)^2,
\end{align*}
where $\sigma_T^2$ denotes the variance of
the centered Gaussian random variable $\frac1T\int_0^TY_s\,ds$, $\mathcal N(0,1)$ is the standard normal random variable.
Therefore, it suffices to show that
\[
\frac{T^{2h_2+2\delta-1}}{\left(x_0+\sigma_T \mathcal N(0,1)\right)^2}\xrightarrow{\P}0\quad\text{as }T\to\infty.
\]
In order to establish this convergence, we will bound $\sigma_T^2$ from below. We have
\begin{align*}
\sigma^2_T&=\frac{1}{T^2}\E\left(\int_0^TY_s\,ds\right)^2
=\frac{1}{T^2}\int_0^T\int_0^T\E Y_sY_u\,du\,ds\\
&=\frac{1}{T^2}\int_0^T\int_0^TD(H_s,H_u)\left(s^{H_s+H_u}+u^{H_s+H_u}-\abs{s-u}^{H_s+H_u}\right)\,du\,ds,
\end{align*}
where $D(x,y)=\frac{\pi}{\Gamma(x+y+1)\sin(\pi(x+y)/2)}$, see~\cite{Ayache-covar} or \cite[p.~213]{StoevTaqqu}.
Further,
\begin{align*}
\sigma^2_T
&=\frac{1}{T^2}\int_0^T\int_0^sD(H_s,H_u)\left(s^{H_s+H_u}+u^{H_s+H_u}-(s-u)^{H_s+H_u}\right)\,du\,ds\\
&\quad+\frac{1}{T^2}\int_0^T\int_s^TD(H_s,H_u)\left(s^{H_s+H_u}+u^{H_s+H_u}-(u-s)^{H_s+H_u}\right)\,du\,ds\\
&\ge\frac{1}{T^2}\left(\int_0^T\int_0^sD(H_s,H_u)u^{H_s+H_u}\,du\,ds
+\int_0^T\int_s^TD(H_s,H_u)s^{H_s+H_u}\,du\,ds\right),
\end{align*}
Since $D(x,y)$ is positive and stays bounded away from 0 for $x,y\in[h_1,h_2]$, we have
\begin{align*}
\sigma^2_T&\ge\frac{C}{T^2}\left(\int_0^T\int_0^su^{H_s+H_u}\,du\,ds
+\int_0^T\int_s^Ts^{H_s+H_u}\,du\,ds\right)\\
&=\frac{2C}{T^2}\int_0^T\int_0^su^{H_s+H_u}\,du\,ds
=\frac{2C}{T^2}\int_0^T\int_0^sT^{H_s+H_u}\left(\frac{u}{T}\right)^{H_s+H_u}\,du\,ds\\
&\ge\frac{2C}{T^2}\int_0^T\int_0^sT^{2h_1}\left(\frac{u}{T}\right)^{2h_2}\,du\,ds
=\frac{2CT^{2h_1}}{(2h_2+1)(2h_2+2)}
= C_1T^{2h_1}
\end{align*}
for $T>1$.

Thus, for any $\eps>0$
\begin{align*}
\P\set{\frac{T^{2h_2+2\delta-1}}{\left(x_0+\sigma_T \mathcal N(0,1)\right)^2}>\eps^2}
&\le\P\set{\abs{\frac{x_0}{\sigma_T}+ \mathcal N(0,1)}<\frac{T^{h_2+\delta-1/2}}{\eps\sigma_T}}\\
&\le\P\set{\abs{\frac{x_0}{\sigma_T}+ \mathcal N(0,1)}<\frac{T^{h_2-h_1+\delta-1/2}}{\eps C_1^{1/2}}}
\to0
\end{align*}
as $T\to\infty$ for $0<\delta<1/2-h_2+h_1$.
\end{proof}


\end{document}